\documentclass[12pt]{amsart}
\DeclareFontFamily{OML}{script}{}{}
\DeclareFontShape{OML}{script}{m}{it}
{ <5-20> rsfs10 }{}
\DeclareMathAlphabet{\mathscript}{OML}{script}{m}{it}

\renewcommand{\mathcal}[1]{{\mathscript #1}\hspace{0.2ex}}
\usepackage{cite}
\usepackage{color}

\usepackage[pagewise]{lineno}

\usepackage{cite}
\ifx\red\undefined
\newcommand{\red}{\color{red}}
\newcommand{\blue}{\color{blue}}

\fi

\usepackage{graphicx,graphics}
\usepackage{pifont}
\usepackage{amsthm}
\usepackage{bm}
\usepackage{amscd}
\usepackage{amsmath}
\usepackage{latexsym}
\usepackage{amsfonts}
\usepackage{amssymb}
\usepackage{color}
\usepackage[notref,final]{showkeys}

\usepackage[ansinew]{inputenc}
\usepackage[encapsulated]{CJK}

\ifx\text\undefined
\newcommand{\text}{\mbox}
\fi
\ifx\operatorname\undefined
\newcommand{\operatorname}{\mathop}
\fi

\newcommand{\ki}{{\mbox{\raise.5ex\hbox{$\chi$}}\hspace{.2ex}}}
\newcommand{\re}[1]{\mbox{\rm$($\ref{#1}$)$}}

\newcommand{\e}{\varepsilon}
\renewcommand{\epsilon}{\varepsilon}
\newcommand{\p}{\partial}

\newcommand{\g}{\gamma}
\renewcommand{\bar}{\overline}
\newcommand{\dis}{\displaystyle}
\newcommand{\la}{\lambda}

\newcommand\be{\begin{equation}}
\newcommand\ee{\end{equation}}
\newcommand\bea{\begin{eqnarray}}
\newcommand\eea{\end{eqnarray}}
\newcommand\beaa{\begin{eqnarray*}}
\newcommand\eeaa{\end{eqnarray*}}

\newcommand{\dist}{\mathrm{dist}}

\renewcommand{\le}{\leqslant}
\renewcommand{\leq}{\leqslant}
\renewcommand{\ge}{\geqslant}
\renewcommand{\geq}{\geqslant}

\newenvironment{eqa}{\begin{equation}%
  \begin{array}{rcl}}{\end{array}\end{equation}}
\newcommand\beqa{\begin{eqa}}
\newcommand\eeqa{\end{eqa}}

\numberwithin{equation}{section}

\newtheorem{thm}{Theorem}[section]
\newtheorem{cor}[thm]{Corollary}
\newtheorem{lem}{Lemma}[section]
\newtheorem{rem}{Remark}[section]
\newtheorem{prop}{Proposition}[section]
\theoremstyle{definition}

\newcommand{\void}[1]{}

\newcommand{\m}{\hspace{1em}}
\newcommand{\mm}{\hspace{2em}}

\begin{document}\begin{CJK}{UTF8}{gkai}

\title[LIOUVILLE-TYPE THEOREMS AND EXISTENCE RESULTS]{LIOUVILLE-TYPE THEOREMS AND EXISTENCE OF SOLUTIONS FOR QUASILINEAR ELLIPTIC EQUATIONS WITH NONLINEAR GRADIENT TERMS}
\author{ Caihong Chang, Bei Hu and Zhengce Zhang }
\date{\today}
\address[Caihong Chang]{School of Mathematics and Statistics, Xi'an Jiaotong University,
Xi'an, 710049, P. R. China}
\email{caihong666@stu.xjtu.edu.cn}
\address[Bei Hu]{Department of Applied and Computational Mathematics and Statistics, University of Notre Dame, Notre Dame, IN 46556, USA}
\email{b1hu@nd.edu}
\address[Zhengce Zhang]{School of Mathematics and Statistics, Xi'an Jiaotong University,
Xi'an, 710049, P. R. China}
\email{zhangzc@mail.xjtu.edu.cn}
\thanks{Corresponding author: Zhengce Zhang}
\thanks{ }

\maketitle
\begin{abstract}
This paper is concerned with two properties of positive weak solutions of quasilinear elliptic equations with nonlinear gradient terms. First, we show a Liouville-type theorem for positive weak solutions of the equation involving the $m$-Laplacian operator
\begin{equation*}
-\Delta_{m}u=u^q|\nabla u|^p\ \ \ \mathrm{in}\ \mathbb{R}^N,
\end{equation*}
where $N\geq1$, $m>1$ and $p,q\geq0$.  The technique of Bernstein gradient estimates
is ultilized to study  the case $p<m$. Moreover, a Liouville-type theorem for supersolutions under subcritial range of exponents
\begin{equation*}
q(N-m)+p(N-1)<N(m-1)
\end{equation*}
is also established.
Then,   we use a degree argument to obtain the existence of positive weak solutions for a nonlinear Dirichlet problem of the type $-\Delta_m u = f(x,u,\nabla u)$, with $f$ satisfying certain structure conditions. Our proof is based on a priori estimates, which will be accomplished by using a blow-up argument together with the Liouville-type theorem   in the half-space.
As another application, some new Harnack inequalities are proved.
\end{abstract}

\maketitle
\tableofcontents
\section{Introduction}
The first primary goal of this paper is to derive Liouville-type theorems for positive weak solutions of the following elliptic equation
\begin{equation}\label{1.1}
-\Delta_{m}u=u^q|\nabla u|^p\ \ \ \mathrm{in}\ \mathbb{R}^N,
\end{equation}
where $N\geq1$, $p,q\geq0$, and $\Delta_{m}u=\mathrm{div}(|\nabla u|^{m-2}\nabla u)$ is the $m$-Laplacian operator with $m>1$.

The background of the equation \eqref{1.1} is rich. Barenblatt {\em et.\ al.} \cite{BBCP} proposed the following equation
\begin{equation}\label{6666666}
\gamma h_t=\kappa \frac{\partial^2(h^2)}{\partial x^2}+\mu\left|\frac{\partial h}{\partial x}\right|^2,\ \ \ h\geq0,
\end{equation}
as a qualitative mathematical model studying the ground water flow in a water-absorbing fissured porous rock in one spacial dimension.  Here $h$ is the groundwater level, $t$ is the time, $x$ is the horizontal space coordinate along the impermeable bed. The parameters $\gamma$, $\kappa$ and $\mu$  characterize the medium and the interaction fluid-rock, which are assumed to be constants. The parameter $\mu$ could be positive or negative according to the rock non-fissured or fissured respectively. Peral \cite{P} analyzed the corresponding porous media equation (i.e., $\alpha>1$) and the fast diffusion equation (i.e., $0<\alpha<1$)
\begin{equation}\label{222}
v_t-\Delta\left(v^{\alpha}\right)=|\nabla v|^p+f(x,t).
\end{equation}
For the stationary problem of \eqref{222}, putting $u=v^{\alpha}$, we get
\begin{equation}\label{0.91}
-\Delta u=u^{p\beta}|\nabla u|^p+\lambda f,\ \ \ \beta=\left(\frac{1}{\alpha}-1\right)\in(-1,\infty).
\end{equation}
Clearly, the mixed nonlinearity of \eqref{0.91} is in the form of the product of a power of $u$ and a power of $|\nabla u|$, which is consistent with the form of nonlinearity in \eqref{1.1}. In the present paper,  Laplacian operator $\Delta$ in \eqref{0.91} is generalized to $m$-Laplacian operator $\Delta_m$ with $m>1$ (i.e., \eqref{1.1}), and more abundant conclusions are obtained.

The Liouville-type theorems go back to the early work \cite{C}, Augustin Cauchy showed that any bounded entire function of a single complex variable must be constant. Protter and Weinberger \cite{PW} generalized this classical result to real harmonic functions which are bounded either from above or below, and superharmonic functions which are bounded from below. More general nonlinearities and operators of the same problem, such as \eqref{1.1}, were studied in some other works. We shall list these works in different cases below.

The case without gradient terms, equation \eqref{1.1}, when $m=2$ and $p=0$, reduces to the famous Lane-Emden equation, i.e.,
\begin{equation}\label{1.2}
-\Delta u=u^q\ \ \ \mathrm{in}\ \mathbb{R}^N.
\end{equation}
In the pioneering reference \cite{GS}, Gidas and Spruck asserted that there is no positive solution of \eqref{1.2} with $N>2$ and $1\leq q<\frac{N+2}{N-2}$ (Sobolev exponent), which is a sharp result. Bidaut-V$\mathrm{\acute{e}}$ron \cite{M} considered \eqref{1.2} on an exterior domain, and obtained the same conclusion with smaller range $1<q<\frac{N}{(N-2)_+}$ (Serrin exponent). Serrin and Zou \cite{SZ} generalized the result to the inequality
\begin{equation*}
-\Delta u\geq u^q\ \ \ \mathrm{in}\ \Omega,
\end{equation*}
where $\Omega\subset\mathbb{R}^N$ is a connected open set and $0<q<\frac{N}{(N-2)_+}$, which is the optimal condition. Recently, Dupaigue {\em et.\ al.} \cite[Theorem 1.1]{DSS} studied the problem in the half-space
\begin{equation}\label{0.111}
\left\{\begin{array}{ll}
-\Delta u=u^q\ \ \ &\mathrm{in}\ \mathbb{R}^N_+\\
u=0\ \ \ &\mathrm{on}\ \partial\mathbb{R}^N_+,
\end{array}\right.
\end{equation}
where
\begin{equation}\label{0.222}
\mathbb{R}^N_+=\{(x_1,...x_{n-1},x_n)\in\mathbb{R}^N,\ x_n>0\}.
\end{equation}
They asserted that the equation \eqref{0.111} does not admit any positive classical solution which is monotone in the normal direction. The conclusion could be regraded as an improvement of \cite[Theorem 1.1]{SS}, which established that there is no positive solution of the equation \eqref{0.111} satisfying the growth condition
\begin{equation*}
u(x)\leq C(1+x_n)^{\frac{2}{q-1}}
\end{equation*}
for some $C>0$. The Liouville-type theorems in $\mathbb{R}^N$ and $\mathbb{R}_+^N$ are also studied for an elliptic system in \cite{PQS}. Some scholars extended the above considerations to degenerate elliptic problems. Bidaut-V$\mathrm{\acute{e}}$ron and Pohozaev \cite{BP} studied the inequality
\begin{equation}\label{A1}
-\Delta_mu\geq u^{p-1}\ \ \ \mathrm{in}\ \Omega,
\end{equation}
where $\Omega$ is an exterior domain. They presented that the equation \eqref{A1} admits only the trivial solution $u\equiv0$, provided $p\in(1,\infty)$ when $N=m$, or $p\in\left(1,\frac{m(N-1)}{N-m}\right)$ when $N>m$, which was first proved by Mitidieri and Pokhozhaev \cite{MP1}. Serrin and Zou \cite{SZ} solved the case $p\in(1,\infty)$ when $N<m$. If $\Omega=\mathbb{R}^N$, the range of exponent $p$ would be larger, see \cite[Theorem II]{SZ}. Meanwhile, \cite{SZ} studied the problems $-\Delta_mu\geq0$ and $\Delta_mu+f(u)=0$ under the suitable assumptions on $f$.

The case with gradient terms, equation \eqref{1.1}, when $m=2$ and $q=0$, reduces to the stationary case of the viscous Hamilton-Jacobi equation, i.e.,
\begin{equation}\label{1.3}
-\Delta u=|\nabla u|^p\ \ \ \mathrm{in}\ \Omega.
\end{equation}
Lions \cite[Corollary IV]{L} asserted that any classical solution of \eqref{1.3} with $p>1$ in $\mathbb{R}^N$ has to be constant. The authors of \cite{PV,FPS} have studied \eqref{1.3} in the half-space \eqref{0.222} with Dirichlet boundary conditions. When $p\in(1,2]$, Porretta and V$\mathrm{\acute{e}}$ron \cite{PV} proved that the solution of \eqref{1.3} depends only on the variable $x_n$. The proof was based on the existence of a finite limit as $x_n\to\infty$. When $p\in(2,\infty)$, Filippucci {\em et.\ al.} \cite[Theorem~1.1]{FPS} obtained a similar result, which was proved by combining a moving planes technique, Bernstein-type estimates and a compactness argument.

For the case with mixed reaction terms involving the product of the function and its gradient, we state some Liouville-type results of supersolutions, firstly. Burgos-P$\acute{\mathrm{e}}$rez {\em et.\ al.} \cite[Corollary 3]{BGQ} studied the inequality
\begin{equation}\label{1.4}
-\Delta u\geq f(u)|\nabla u|^p\ \ \ \mathrm{in}\ \mathbb{R}^N
\end{equation}
under the assumptions
\begin{equation}\label{1.5}
\left\{\begin{array}{ll}
N\geq3,\ 0<p\leq\frac{N}{N-1},\ f\in C\left([0,+\infty)\right)\ \mathrm{is\ positive\ in}\ (0,\infty),\\
\int_{0}^{\delta}\frac{f(t)}{t^{\theta}}\mathrm{dt}<+\infty\ \mathrm{does\ not\ hold\ for\ some}\ \delta>0\ \mathrm{and}\ \theta=\frac{(2-p)(N+1)}{N-2},
\end{array}\right.
\end{equation}
and obtained that the only positive solution of \eqref{1.4} is constant. For the special case $f(u)=u^q$, the inequality \eqref{1.4} reduces to
 \begin{equation}\label{1.7}
-\Delta u\geq u^q|\nabla u|^p\ \ \ \mathrm{in}\ \mathbb{R}^N.
\end{equation}
Using $f(u)=u^q$ to verify conditions which $f$ satisfies in \eqref{1.5}, one can deduce the range (subcritical range of exponents)
$$q(N-2)+p(N-1)<N.$$
Indeed, when $f(u)=u^q$, the Liouville-type theorem of \cite[Theorem 2.1]{BGV} corresponds to \cite[Corollary 3]{BGQ}. Meanwhile, for \eqref{1.7} with
$$q(N-2)+p(N-1)\leq N,\ \ p+q>1,$$
any positive solution of \eqref{1.7} must be constant, see \cite[Theorem 7.1]{CM}. However, for \eqref{1.7} with
$$q(N-2)+p(N-1)>N,$$
there exists a positive, nonconstant and bounded classical solution of   the form \mbox{$C(1+|x|^2)^{-\beta}$} for suitable $\beta, C>0$. For degenerate elliptic inequalities, we refer the readers to the survey works \cite{MP,F1}. Mitidieri and Pokhozhaev \cite[Theorem 15.1]{MP} considered the problem
\begin{equation}\label{MP}
\left\{\begin{array}{ll}
-\Delta_mu\geq f(u,\nabla u),&x\in\mathbb{R}^N,\\
u\geq0, &x\in\mathbb{R}^N,
\end{array}\right.
\end{equation}
and they showed that this problem admits no other solution except constants under the assumptions
\begin{equation*}
\left\{\begin{array}{ll}
f\ \mathrm{is\ a\ Carath\acute{e}odory\ function\ such\ that}\ f(t,s)\geq K_0t^q|s|^p,\\
\forall\ (t,s)\in\mathbb{R}_+\times\mathbb{R}^N, \mathrm{where}\ 1<m<N,\ q>0,\ p\geq0,\ p+q>m-1,\ K_0>0,\\
q(N-m)+p(N-1)<N(m-1)\ (\mathrm{subcritical\ range\ of\ exponents}).
\end{array}\right.
\end{equation*}
Indeed, \cite[Theorem 15.1]{MP} corresponds to \cite[Corollary 1]{F1}.

Secondly, we state some Liouville-type results of solutions. Serrin \cite{S} considered that if $u$ is an entire solution of the equation
$$-\Delta u=f(u,\nabla u)\ \ \ \mathrm{in}\ \mathbb{R}^N$$
under the assumptions: $\frac{\partial f}{\partial u}\leq0$, both $u$ and $\nabla u$ are bounded, then $u$ must be constant. Furthermore, it is established in \cite{PS} that $\nabla u$ is necessarily bounded in $\mathbb{R}^N$. From this fact, Caffarelli {\em et.\ al.} \cite{CGS} got a standard Liouville-type result for bounded solutions. Bidaut-V$\mathrm{\acute{e}}$ron {\em et.\ al.} \cite[Corollary B-1]{BGV} studied the equation
\begin{equation}\label{1.6}
-\Delta u=u^q|\nabla u|^p\ \ \ \mathrm{in}\ \mathbb{R}^N,
\end{equation}
where $N\geq2$, $0\leq p<2$ and $q\geq0$ satisfy some suitable restrictions. By using a direct Bernstein method, they obtained that any positive solution is a constant. Filippucci {\em et.\ al.} \cite[Theorem 1.1]{FPS1} also considered \eqref{1.6} for any $p>2$ and $q>0$ (or $q>1-p$). They asserted that any positive bounded solution is constant. The proof was based on a local Bernstein argument and monotonicity properties for the spherical averages of sub-harmonic and super-harmonic functions. Note that the result holds without boundedness assumption in the radial case, see \cite[Theorem C]{BGV}. Meanwhile, some extensions to elliptic systems were   given in \cite{FPS1}. Recently, Bidaut-V$\mathrm{\acute{e}}$ron \cite{B} proved that Liouville-type results are also true for \eqref{1.6} ($p>2$) and \eqref{1.1} ($p>m$) without the assumption of boundedness on the solutions. The proof was based on Bernstein estimates and an Osserman's type estimate for the inequality satisfied by the gradient. Motivated by \cite{BGV}, the present paper extends the Liouville-type theorem to degenerate elliptic equation \eqref{1.1} under the assumption $p<m$.

Our main results are stated as follows.

\begin{thm}\label{thm:1.2}
Let $N\geq2$, $m>1$, $0\leq p<m$ and $q\geq0$ such that $p+q-m+1>0$. If $u$ is a positive weak solution of \eqref{1.1} in $\mathbb{R}^N$ and one of the following assumptions is fulfilled,

$\mathrm{(i)}$ $q\geq1$ and $p+q-m+1<\frac{4(m-1)}{N}$,

$\mathrm{(ii)}$ $0\leq q<1$ and $p+q-m+1<\frac{(m-1)(q+1)^2}{qN}$.

\noindent Then $u$ is constant.
\end{thm}

\begin{rem}
$\mathrm{(i)}$ For the case $p=m$, Bidaut-V$\mathrm{\acute{e}}$ron \cite[Theorem 2.1]{B} (also see \cite[Subsection 3.1]{HB}) found a change of variable
$$v(x)=\int_0^{u(x)}\exp\Big\{\frac{1}{(q+1)(m-1)}s^{q+1}\Big\}\mathrm{ds}$$
to deduce that $\Delta_mv=0$, and then use the classical Liouville-type theorem (see \cite[Theorem II(i)]{SZ}) to obtain that any positive weak solution of \eqref{1.1} is constant.

$\mathrm{(ii)}$ The Liouville-type result theorem \ref{thm:1.2} would not hold without restrictions.
In fact, some restrictions are needed on the parameters $N,m,p,q$ (i.e., $p+q-m+1>0$ and either $\mathrm{(i)}$ or $\mathrm{(ii)}$). These conditions are needed from a computation of a discriminant of quadratic function, see Step 3 of Lemma \ref{lem3.2} for more details.

$\mathrm{(iii)}$  The appearance of the term $|\nabla u|^p$ of \eqref{1.1} will lead to positive constant solutions, which makes it difficult to apply this theorem to prove the existence of positive solutions.

$\mathrm{(iv)}$ Bidaut-V$\mathrm{\acute{e}}$ron \cite[Theorem B]{BGV} made the transformations: $u=v^{-\gamma}$, $z=|\nabla v|^2$, $z=v^{-\lambda}Y$, and then applied Osserman's type estimate (Lemma \ref{lem3.1}) on $Y$. We notice that
\[
Y=v^{\lambda}z =v^{\lambda}|\nabla v|^2
=\Big(\frac{\lambda}{2}+1\Big)^{-2}\left|\nabla v^{(\frac{\lambda}{2}+1)}\right|^2
=\Big(\frac{\lambda}{2}+1\Big)^{-2}\left|\nabla u^{-\frac1{\gamma}(\frac{\lambda}{2}+1)}\right|^2.
\]
Thus we shall take $\Big(-\frac1{\gamma}(\frac{\lambda}{2}+1)\Big)^{-1}$ as a new parameter $\beta$, and make the transformations: $u=v^{\beta}$, $z=|\nabla v|^2$. Compared with \cite[Theorem B]{BGV}, the calculation can be simplified by applying Osserman's type estimate on $z$.
\end{rem}

The following result is a Liouville-type theorem for supersolutions of \eqref{1.1}.
\begin{thm}\label{thm:1.4}
Assume $N\geq2$, $1<m<N$, $p\geq0$, $q\geq0$ and
\begin{equation}\label{s0}
q(N-m)+p(N-1)<N(m-1).
\end{equation}
Let $u$ be a positive solution of the inequality
\begin{equation}\label{s}
-\Delta_mu\geq u^q|\nabla u|^p\ \ \ \mathrm{in}\ \mathbb{R}^N.
\end{equation}
Then $u$ is a constant.
\end{thm}

\begin{rem}
$\mathrm{(i)}$ The result was obtained first under the case $p+q-m+1>0$ by Mitidieri and Pohozaev in \cite{MP}, and also in the case $p+q-m+1\leq0$ by Filippucci in \cite[Corollaries 1.5 and 1.6]{F2} (the only case not considered was $q=0$). A simple proof is presented in this paper extending the one of \cite[Theorem 2.1]{BGV}.

$\mathrm{(ii)}$ We note that the Liouville-type theorem still holds for the more general equation
\begin{equation*}
-\mathrm{div}\mathcal{A}(x,u,\nabla u)\geq\mathcal{B}(x,u,\nabla u)\ \ \ \mathrm{in}\ \mathbb{R}^N,
\end{equation*}
where the operators $\mathcal{A}$ and $\mathcal{B}$ satisfy
\begin{equation*}
\langle\mathcal{A}(x,r,\xi),\xi\rangle\geq|\xi|^m
\end{equation*}
and
\begin{equation*}
\mathcal{B}(x,r,\xi)r\leq cr^q|\xi|^p
\end{equation*}
under suitable restrictions on parameters $N,m,p,q$.
\end{rem}

The case $m=2$ in \re{1.1} was studied thoroughly in \cite{BGV}.
This provided us a significant reference to solve problems in the case $m>1$. Bernstein-type estimates are adopted in this paper, which would result the gradient upper bound in the form of the negative exponents of the distance to the boundary (i.e., $\left[\mathrm{dist}(x,\partial\Omega)\right]^{-\beta}$ with $\beta>0$) in a bounded domain $\Omega$, see Lemma \ref{lem3.2}. Since $\mathrm{dist}(x,\partial\Omega)$ can be chosen arbitrarily large when the solution is defined on the entire $\mathbb{R}^N$, the upper bound becomes zero. So the solution is a constant. For the case $|\nabla u|=0$, the gradient estimate is obviously true. In this paper, we primarily  focuses on the case $|\nabla u|>0$, i.e., nondegenerate case. Due to the fact that
\begin{equation*}
\Delta_mu=|\nabla u|^{m-2}\Delta u+(m-2)|\nabla u|^{m-4}\langle D^2u\nabla u,\nabla u\rangle,
\end{equation*}
we divide the calculation of the term $\Delta_mu$ into two parts, namely, the terms $|\nabla u|^{m-2}\Delta u$ and $(m-2)|\nabla u|^{m-4}\langle D^2u\nabla u,\nabla u\rangle$, respectively. The techniques we used in Bernstein-type estimates are as follows.

First, we adopt the method of \cite{BGV} to deal with the calculation originated from the term $|\nabla u|^{m-2}\Delta u$.
The technical challenge lies in that the appearance of the term $|\nabla u|^{m-2}$ complicates each resulting term of estimates and then causing technical difficulties to the calculation.   It's worth noting that the results attributed to $|\nabla u|^{m-2}\Delta u$ are consistent with the results attributed to $\Delta u$ (i.e., the case $m=2$) in \cite{BGV}. In fact, the term $\Delta_mu$ is divided into two terms because the calculation of the term $|\nabla u|^{m-2}\Delta u$ would directly associated with the formulas of the  Laplacian operator $\Delta$, such as
\begin{equation*}
\Delta(u_1u_2)=u_1\Delta u_2+u_2\Delta u_1+2\nabla u_1\cdot\nabla u_2
\end{equation*}
and
\begin{equation*}
\Delta(|\nabla u|^2)=2\langle\nabla(\Delta u), \nabla u\rangle+2|D^2u|^2.
\end{equation*}
Clearly, the terms $\Delta_m(u_1u_2)$ and $\Delta_m(|\nabla u|^2)$ with $m>1$ are more complex to calculate.

Second, the term $(m-2)|\nabla u|^{m-4}\langle D^2u\nabla u,\nabla u\rangle$ will generate a lot of terms in the calculation process, among which the terms containing the second order derivative will be treated as the operator terms (i.e., \eqref{2.19a} and Remark \ref{rem:2.1}), and the other ones are transformed by using the fundamental inequalities and then estimated, see Step 2 of Lemma \ref{lem3.2} for more details.

We recall that one of the most frequent and fundamental applications of Liouville-type theorems is to get a priori estimates for positive solutions of elliptic equations in bounded domains, where topological methods can be used to obtain existence results. Motivated by this idea, the second primary goal of this paper is to prove the existence of positive weak solutions of the following problem
\begin{equation}\label{P}
\left\{\begin{array}{ll}
-\Delta_mu=f(x,u,\nabla u) &\mathrm{in}\ \Omega,\\
u(x)=0 &\mathrm{on}\ \partial\Omega,
\end{array}\right.
\end{equation}
where $\Omega\subset\mathbb R^N$ is a bounded smooth domain, $1<m<N$ and $f:\bar\Omega\times\mathbb R\times\mathbb R^N\to\mathbb R$ satisfies
\bea
\mm\left\{ \begin{split}
&f(x,u,\eta) = f_1(x,u,\eta) + f_2(x,u)+ f_3(x,\eta), \; \forall(x,u,\eta)\in\bar\Omega\times\mathbb R\times\mathbb R^N,\\
& \left\{
  \begin{split}
 &f_1  \text{ is  continuous on } (\Omega\setminus \mathscript{Z})\times R \times R^N \text{ for a zero measure set $\mathscript{Z}\subset \Omega$,}\\
 & bu^q|\eta|^p \le f_1(x,u,\eta) \leq bc_0u^q|\eta|^p \text{ for } u\ge0, \; \text{ $b$ is either $0$ or $1$},
 \end{split}  \right.\\
&  \left\{\begin{split}
 &f_3  \text{ is  continuous on } (\Omega\setminus \mathscript{Z})\times R^N \text{ for a zero measure set $\mathscript{Z}\subset \Omega$,}\\
 &|f_3(x,\eta)| \le M_2 |\eta|^{\alpha_2},
 \end{split}\right.\\
&\left\{\begin{split}
 &u^{\alpha_1} \le f_2(x,u) \le M_1 u^{\alpha_1} \text{ for } u\ge 0, \\
 &|f_2(x_1,u)-f_2(x_2,u)| \le \omega(|x_1-x_2|) u^{\alpha_1}, \m \omega(0+)=0,\\
 & u\frac{\p f_2(x,u)}{\p u} \le \g_2 f_2(x,u) \m\text{for some } \gamma_2\in\left(0,\frac{N(m-1)+m}{N-m}\right), \\
 & u^2 \left|\frac{\p^2 f_2(x,u)}{\p u^2} \right| \le C f_2(x,u),
 \end{split}\right.\\
&\text{where}\ c_0, \; M_1\geq1, \; M_2\geq 0, \; 1<m<N, \;
q\ge 0, \; m-1<p<m,\\
&\frac{p+mq}{m-p}<\alpha_1<\frac{N(m-1)}{N-m}, \; m-1<\alpha_2<\frac{m\alpha_1}{\alpha_1+1},\\
& q(N-m)+p(N-1)<N(m-1).
\end{split} \right. \label{F}
\eea
Notice that we have assumed a structure condition but {\em does not require $f$ to be non-negative.} Continuous differentiability is assumed only on $f_2$; the functions $f_1$ and $f_3$ are not even assumed to be continuous, but our assumptions ensure that $f$ is a Carath\'eodory function. Cases such as
\be \left\{ \begin{split} & f(x, u, \eta) =  b_1(x,u,\eta) u^q |\eta|^p + b_2(x) u^{\alpha_1}
 + b_3(x,\eta)  |\eta|^{\alpha_2}, \\
 & 1\le b_1(x,u,\eta), b_2(x) \le M, \m |b_3(x,\eta)|\le M, \m m, p, q, \alpha_1, \alpha_2 \text{ are as above},\\
 & \text{$b_2(x)$ is uniformly continuous on $\bar\Omega$},
 \end{split}\right.
 \ee are included.

Due to appearance of the term $|\nabla u|$, problem \eqref{P} does not possess a variational structure. Therefore, in the present paper  the topological methods will be used to prove the existence of positive weak solutions.

During the past several decades, an abundance of works regrading existence results of \eqref{P} have emerged in many literatures. For the Laplacian case, we refer to \cite{BT,DLN,DGR,DY,F,GP,GR,WD}. De Figueiredo {\em et.\ al.} \cite{DLN} studied the case without gradient terms: $f(x,u,\nabla u)=f(u)$, where $f$ grows superlinearly but less than critical growth, i.e.,
$$\lim_{u\to\infty}\frac{f(u)}{u^{\beta}}=0,\ \ \ \beta=\frac{N+2}{N-2}.$$
Brezis and Turner \cite{BT} considered the case that $f$ satisfies stronger conditions than \eqref{F}, that is,
\begin{equation}\label{11111111}
\liminf_{u\to\infty}\frac{f(x,u,\eta)}{u}>\lambda_1,\ \lim_{u\to\infty}\frac{f(x,u,\eta)}{u^{\alpha}}=0,\ \limsup_{u\to0}\frac{f(x,u,\eta)}{u}<\lambda_1
\end{equation}
with $\alpha=\frac{N+1}{N-1}$, and the condition \eqref{11111111} holds uniformly for $x\in\overline{\Omega}$, $\eta\in\mathbb{R}^N$. Ghergu and R$\breve{\mathrm{a}}$dulescu \cite{GR} focused on the case
$$f(x,u,\nabla u)=\lambda f(x,u)-K(x)g(u)+|\nabla u|^{\alpha},$$
where $\lambda>0$, $0<\alpha\leq2$, $K\in C^{0,\gamma}(\overline{\Omega})$ with $\gamma\in(0,1)$, $g$ is of a singular nonlinearility, $f$ is smooth and grows sublinearly. A similar model is also considered in \cite{DGR}. In addition, \cite{GP} is the first work dealing with Neumann problems of \eqref{P}. The existence is also studied for the cases of uniformly elliptic operators \cite{WD} and systems of equations \cite{DY,F}.

For the $m$-Laplacian case, Cl$\mathrm{\grave{e}}$ment {\em et.\ al.} \cite{CMM} considered positive radial solutions for a system of the $m$-Laplacian case without the gradient terms. Azizieh and Cl$\acute{\mathrm{e}}$ment \cite{AC} studied the case $f(x,u,\nabla u)=f(u)$, and under the assumptions: $1<m\leq2$, $\Omega$ is convex, $f$ is continuous and $f(0)=0$. Motreanu and Tanaka \cite{MT} considered \eqref{P} and developed an approach based on approximate solutions and a new strong maximum principle. Faraci {\em et.\ al.} \cite{FMP} assumed that $f$ satifies a $(m-1)$-sublinear growth in $u$ and $\nabla u$. The proof was based on subsolution and supersolution techniques, Schaefer's fixed point theorem, regularity results and strong maximum principle.  Zou \cite[Theorem 1.3]{Z} required $f(x,u,\nabla u)$ to satisfy a growth-limit condition, a positivity condition, and to be "super-linear" at the origin.
The result is obtained by combining a fixed point theorem with a priori estimates. Ruiz \cite{R} focused on the case
\begin{equation}\label{0.444}
u^{\delta}-M|\eta|^{p}\leq f(x,u,\eta)\leq c_0u^{\delta}+M|\eta|^{p},
\end{equation}
where $c_0\geq 1$, $M>0$, $p\in\left(m-1,\frac{m\delta}{\delta+1}\right)$ and $\delta\in\left(m-1,m_*-1\right)$ with $m_*=\frac{m(N-1)}{N-m}$. Lorca and Ubilla \cite{LU} generalized \eqref{0.444} to the case
\begin{equation}\label{0.333}
u^{\gamma}-M|\eta|^{p}\leq f(x,u,\eta)\leq c_0u^{\delta}+M|\eta|^{p}
\end{equation}
with $m-1<\gamma\leq\delta<m_*-1$. Recently, in \cite{FL1} and \cite{FL2}, Filippucci and Lini considered the case that $f$ involves an explicit dependence on the solution $u$ in the gradient terms, i.e.,
\begin{equation}\label{0.999}
u^{\delta}-Mu^q|\eta|^{p}\leq f(x,u,\eta)\leq c_0u^{\delta}+Mu^q|\eta|^{p}
\end{equation}
and
$$\max\left\{0,u^{\gamma}-Mu^q|\eta|^{p}\right\}\leq f(x,u,\eta)\leq c_0u^{\delta}+Mu^q|\eta|^{p}$$
under suitable assumptions on the parameters $\delta,p,q,\gamma,c_0,M$, respectively.
In contrast, our result does not require $f$ to be non-negative.

In order to prove the existence of positive weak solutions for \eqref{P}, the authors of \cite{FL2,LU} applied the Rabinowitz type theorem due to \cite{AC}.
It's worth noting  that the results were proved without the help from a Liouville-type theorem for the limit problems, which were obtained via a blow-up argument; for the semilinear case see Gidas and Spruck in \cite{GS,GS1}.
Indeed, in \cite{GS1}, the study of a priori bounds for positive weak solutions of the equation
\begin{equation*}
\left\{\begin{array}{ll}
-\Delta u=f(x,u)&\mathrm{in}\ \Omega,\\
u=0 &\mathrm{on}\ \partial{\Omega}
\end{array}\right.
\end{equation*}
is turned into the study of Liouville-type theorems for the two equations
\begin{equation}\label{0.15}
-\Delta u=u^q\ \ \ \mathrm{in}\ \mathbb{R}^N
\end{equation}
and
\begin{equation}\label{0.16}
\left\{\begin{array}{ll}
-\Delta u=u^q &\mathrm{in}\ \mathbb{R}_+^N,\\
u(x)=0 &\mathrm{on}\ \partial\mathbb{R}_+^N,
\end{array}\right.
\end{equation}
where $q$ is related to a growth condition on $f(x,u)$ with respect to $u$. Gidas and Spruck made the blow-up procedure around the points $x_n$ in which solutions $u_n$ attain their maxima. Assuming $\lim_{n\to\infty}x_n=x_0$, they provided a solution on the entire $\mathbb{R}^N$ if $x_0\in\Omega$ (corresponding to \eqref{0.15}) and a solution in the half-space $\mathbb{R}_+^N$ if $x_0\in\partial\Omega$ (corresponding to \eqref{0.16}). Some scholars intended to extend above considerations to the $m$-Laplacian case. Due to the absence of a Liouville-type theorem in the half-space at that time, they made some restrictions on the domain, see \cite{AC,CMM}. 
Whereafter, Zou \cite{Z} provided a Liouville-type theorem for the $m$-Laplace equation with no gradient terms in the half-space. By using this advantageous tool, Baldelli and Filippucci \cite{BF} recently considered \eqref{P} with
$$f(x,u,\nabla u)\leq u^q+C(|u|^s+|\nabla u|^{\theta})$$
under restrictions on parameters $m,q,s,\theta$. This approach can also be applied in nonlocal elliptic problem, see \cite{BPGQ}. In addition, some qualitative properties of semilinear and quasilinear parabolic problems with nonlinear gradient terms were also studied in \cite{AS,FPS,LS,LZZ,LZZ1,JSS,PS2,SZ1,ZL,ZL1,ZL2,ZZ} and the references therein.

In the present paper,  Theorem \ref{thm:1.4} will be applied to study the existence results of the equation \eqref{P}, the nonlinearity $f$ of which involves the positive term $u^q|\nabla u|^p$ due to \eqref{F}. As mentioned earilier, the application of Harnack inequalities resulted some restrictions on the parameters, namely, \eqref{F}.

Our existence theorem is as follows.
\begin{thm}\label{thm:E}
Assume $1<m<N$. Let $f\in C(\Omega\times\mathbb R\times\mathbb R^N)$ be a nonnegative function verifying \eqref{F}. The problem \eqref{P} admits at least one positive weak solution.
\end{thm}
\begin{rem}\label{remE}
$\mathrm{(i)}$ The assumptions on $f$ in this paper are consistent with those  in\cite{R}, namely,
\begin{equation}\label{T30}
u^{\delta}-M|\eta|^{\alpha_2}\leq f(x,u,\eta)\leq c_0u^{\delta}+M|\eta|^{\alpha_2}.
\end{equation}
where $\alpha_2\in\left(m-1,\frac{m\delta}{\delta+1}\right)$ and $\delta\in\left(m-1,m_*-1\right)$ with $m_*=\frac{m(N-1)}{N-m}$. The parameter $\delta$ in \cite{R} is equal to the parameter $\alpha_1$ in this paper. Note that $\alpha_1\in\left(\frac{p+mq}{m-p},m_*-1\right)$. Compared with \eqref{T30}, the nonlinearity $f$ in this paper contains one more term of  the prototype $c_3(x)u^q|\eta|^p$, $1\le c_3(x)\le c_0$, which
is not suitable for application of the Young's inequality for further simplification. In fact,
\begin{equation}\label{Z1}
\begin{split}
u^q|\eta|^p&\leq\frac{q(m-p)}{p+mq}u^{\frac{p+mq}{m-p}}+\frac{p(q+1)}{p+mq}|\eta|^{\frac{p+mq}{q+1}}\\
&=\frac{q(m-p)}{p+mq}u^{\alpha_1}+\frac{p(q+1)}{p+mq}|\eta|^{\frac{m\alpha_1}{\alpha_1+1}}.
\end{split}
\end{equation}
The Young's inequality enlarges the powers of $u$ and $|\eta|$ (i.e., $\alpha_1>q$, $\frac{m\alpha_1}{\alpha_1+1}>p$), which makes the ranges of the parameters in \eqref{F} inaccurate. In addition, the assumption $\alpha_2\in\left(m-1,\frac{m\alpha_1}{\alpha_1+1}\right)$ is fundamental in \cite{R}. It can be seen that $\frac{m\alpha_1}{\alpha_1+1}$ in \eqref{Z1} is the border line case and is not covered by results in \cite{R}.

$\mathrm{(ii)}$ Argue as in \cite{FL2,LU}, if $f$ satisfies the following condition
\begin{equation*}
\max\left\{0,u^{\gamma_1}|\eta|^{\gamma_2}+M_1u^{\gamma_3}-M_2|\eta|^{\alpha_2}\right\}\leq f(x,u,\eta)\leq c_0u^q|\eta|^p+M_1u^{\alpha_1}+M_2|\eta|^{\alpha_2}
\end{equation*}
with some restrictions on the parameters $N,m,p,q,\gamma_1,\gamma_2,\gamma_3,\alpha_1,\alpha_2$, Theorem \ref{thm:E} still holds.
\end{rem}

Due to the appearance of the term $u^q|\nabla u|^p$, we adopt the following techniques in proving the existence of positive weak solutions.

First, we state the techniques in the proof of the Harnack inequality of weak solutions to the inequality
\begin{equation}\label{0.777}
bu^q|\nabla u|^p+ u^{\alpha_1}-M_2|\nabla u|^{\alpha_2}\leq-\Delta_mu\leq c_0bu^q|\nabla u|^p+M_1u^{\alpha_1}+M_2|\nabla u|^{\alpha_2}+\lambda,
\end{equation}
where $b=0$ or $1$, see Theorem \ref{thm:5.3}. In the present paper, we apply the classical Harnack inequality in \cite[Lemma 2.2]{R}, which is for the inequality
\begin{equation}\label{0.888}
|\Delta_mu|\leq c(x)|\nabla u|^{m-1}+d(x)u^{m-1}+f(x)
\end{equation}
under additional restrictions on functions $c(x)$, $d(x)$ and $f(x)$.
It is an interesting question  whether the classical results  \eqref{0.777} and \eqref{0.888}  can be applied directly  on the term $u^q|\nabla u|^p$. In fact, since
$$u^q|\nabla u|^p=u^q|\nabla u|^{p-m+1}|\nabla u|^{m-1},$$
the term $u^q|\nabla u|^{p-m+1}$ can be viewed as the coefficient of the term $|\nabla u|^{m-1}$. We can also write
$$u^q|\nabla u|^p=u^{q-m+1}|\nabla u|^{p}u^{m-1},$$
the term $u^{q-m+1}|\nabla u|^{p}$ also can be viewed as the coefficient of the term $u^{m-1}$. Thus, according to \eqref{0.777}, we can find the detailed expressions of $c(x)$, $d(x)$ and $f(x)$ in \eqref{0.888}, and then only need to substitute the conditions satisfied by $c(x)$, $d(x)$, $f(x)$ for verification, see \eqref{5.12}-\eqref{5.14}. In particular, in the calculation of the term $\int_{B_{2R}}\left[u^q|\nabla u|^{p-m+1}\right]^{\sigma'}$ during validation (i.e., \eqref{0.666}), we use the H\"{o}lder inequality to form the integral expressions \eqref{5.2} and \eqref{5.11}. Clearly, the integrand of \eqref{5.2} is a power of $u$, the integrand of \eqref{5.11} is the product of a power of $u$ and a power of $|\nabla u|$, which is related to the nonlinearity in \eqref{1.1}. \eqref{5.11} is obtained via selecting an appropriate test function and then using the fundamental inequalities. Moreover, in the integral estimates, we give more accurate ranges of exponents of $u,|\nabla u|$ in Lemma \ref{lem5.1} than those in \cite[Lemma 2.1]{R}.

Second, we consider more general nonlinearity $f$ in this paper, and does not require $f$ to be non-negative by selecting an appropriate test function, see Lemma \ref{lem111} and Proposition \ref{prop6.1}. In the blow-up process, we make the transformation $w_n=T(u_n)$, and find that $w_n$ satisfy
\[
-\Delta_m w_n(y)=\theta_n(y,w_n,\nabla w_n).
\]
Further estimates on $\theta_n$ show that the limit of $\theta_n$ is independent of $y$. From the regularity and the Harnack inequality given by Theorem \ref{thm:5.3}, we deduce that $w_n\to w$ on any compact subset of $\bar { \mathbb{R}^N_+} \cap \bar \Omega_n$, and $w$ satisfies
\[
- \Delta_m w  =  B(w),  \m y\in \mathbb{R}^N_+,
\]
where
\beaa
&& w^{\alpha_1} \le B(w) \le M_1 w^{\alpha_1} , \mm
 w \frac{\p B(w)}{\p w} \le \gamma_2 B(w).
\eeaa
The Liouville-type theorem in the half-space in \cite[Theorem 1.1]{Z} implies that $w\equiv0$, which is a contradiction with the maximun principle, given by Theorem \ref{thm:6.1}.


The paper is organized as follows. In Section 2, we present Liouville-type theorems \ref{thm:1.2} and \ref{thm:1.4}. In Section 3, we prove Theorem \ref{thm:E}. In Section 4, we point out a typographic error of a weak Harnack inequality frequently
quoted in connection with this type of estimates.

\section{Liouville-type theorems}

\subsection{The case $p<m$}
\paragraph{\ \ \ \ The following result will be useful in the proof of Lemma \ref{lem3.2}.}
The proof is referred to   part of the proof taken from (2.11) to (2.12) of \cite[Proposition 2.1]{BGV1}.
\begin{lem}\label{lem3.1}
Let $\xi>1$ and $R>0$. Assume $v$ is continuous,   $Y$ is continuous and nonnegative on $\overline{B_R}$ and $C^1$ on the set $S=\left\{x\in B_R:Y(x)>0\right\}$, and
$|\nabla v|>0$ on $S$. For some real number $d$, if $Y$ satisfies
\begin{equation*}
-\Delta Y-(m-2)\frac{\langle D^2Y\nabla v, \nabla v\rangle}{|\nabla v|^2}+Y^{\xi}-d\frac{|\nabla Y|^2}{Y}\leq0
\end{equation*}
on each connected component of $S$, then
\begin{equation*}
Y(0)\leq C_{N,\xi,d}R^{-\frac{2}{\xi-1}}.
\end{equation*}
\end{lem}

\begin{rem}\label{rem:2.1}
$\mathrm{(i)}$ Denote the operator $\mathcal{A}$ by
\begin{equation*}
\begin{split}
Y\rightarrow\mathcal{A}Y:=&-\Delta Y-(m-2)\frac{\langle D^2Y\nabla v,\nabla v\rangle}{|\nabla v|^2}\\
=&-\sum_{i,j=1}^{N}\left(\delta_{ij}+(m-2)\frac{v_iv_j}{|\nabla v|^2}\right)Y_{ij}=-\sum_{i,j=1}^{N}a_{ij}Y_{ij},
\end{split}
\end{equation*}
where $\delta_{ij}=1$ if $i=j$, $\delta_{ij}=0$ if $i\neq j$, $v_i=\frac{\partial v}{\partial x_i}$, $Y_{ij}=\frac{\partial^2Y}{\partial x_i\partial x_j}$ and $a_{ij}$ depend on $\nabla v$, then
\begin{equation*}
\mathrm{\min}\{1,m-1\}|\xi|^2\leq\sum_{i,j=1}^{N}a_{ij}\xi_{i}\xi_{j}\leq\mathrm{\max}\{1,m-1\}|\xi|^2
\end{equation*}
for all $\xi=(\xi_1,...\xi_N)\in\mathbb R^N$. Consequently, $\mathcal{A}$ is uniformly elliptic if $|\nabla v|>0$.\\
$\mathrm{(ii)}$ The result of \cite[Lemma 3.1]{B} extended the results of \cite[Proposion 2.1]{BGV1} and \cite[Lemma 2.2]{BGV}, and implied an Osserman's type property of inequality. Namely, suppose
\begin{equation*}
\mathcal{A}Y+\alpha(x)Y^{\xi}-\beta(x)-d\frac{|\nabla Y|^2}{Y}\leq0,
\end{equation*}
where $\xi>1$, $d=d(N,p,q)$, $\alpha,\beta$ are continuous in $\Omega$ and $\alpha$ is positive. Then there exists a constant $C=C(N,p,q,\xi)>0$ such that for any ball $\overline{B}(x_0,\rho)\subset\Omega$ it holds
\begin{equation*}
Y(x_0)\leq C\left[\frac{1}{\rho^2}\max_{B_{\rho}(x_0)}\frac{1}{\alpha}\right]^{\frac{1}{\xi-1}}+\left[\max_{B_{\rho}(x_0)}\frac{\beta}{\alpha}\right]^{\frac{1}{\xi}}.
\end{equation*}
The above result plays an vital role in proving a Liouville-type theorem for \eqref{1.1} without the assumption of boundedness on the solution when $p>m$.
\end{rem}

The Liouville-type theorem is based on the following Bernstein-type estimate, which can be referred to the series works \cite{AS,CJZ,FPS,L,LS,PS1,ZH,ZH1}.

\begin{lem}\label{lem3.2}
Under the assumptions of Theorem \ref{thm:1.2}, if $u$ is a weak solution of \eqref{1.1} in $B_R$, there exist positive constants   $\alpha=\alpha(N,m,p,q)$ 
and $C=C(N,m,p,q)$ such that
\begin{equation*}
|\nabla u^{\alpha}(0)|\leq CR^{-1-\alpha\frac{m-p}{p+q-m+1}}.
\end{equation*}
\end{lem}
\begin{proof}
We assume  $|\nabla u|\not\equiv 0$, and   consider the region where $|\nabla u|>0$.

\textbf{Step\ 1}. Transformation of the equation \eqref{1.1}.
Set $u=v^{\beta}$, where   $\beta>\max\left(0,\frac{p-m+1}{p+q-m+1}\right)$.
We compute
\begin{eqnarray}\label{2.1}
 && \nabla u  =  \beta v^{\beta-1}\nabla v, \\
\label{2.2}
 && \Delta u =\beta v^{\beta-1}\Delta v+\beta(\beta-1)v^{\beta-2}|\nabla v|^2
\end{eqnarray}
and
\begin{equation}\label{2.3}
D^2u=\beta v^{\beta-1}D^2v+\beta(\beta-1)v^{\beta-2}(\nabla v)^t\nabla v,
\end{equation}
where $(\nabla v)^t$ is the transpose of $\nabla v$. It follows from \eqref{2.1} and \eqref{2.3} that
\begin{equation}\label{2.4}
\begin{split}
\langle D^2u\nabla u,\nabla u\rangle
=&\langle\beta v^{\beta-1}D^2v(\beta v^{\beta-1}\nabla v),\beta v^{\beta-1}\nabla v\rangle\\
&+\langle\beta(\beta-1)v^{\beta-2}(\nabla v)^t\nabla v(\beta v^{\beta-1}\nabla v),\beta v^{\beta-1}\nabla v\rangle\\
=&\beta^3v^{3\beta-3}\langle D^2v\nabla v,\nabla v\rangle+\beta^3(\beta-1)v^{3\beta-4}|\nabla v|^4.
\end{split}
\end{equation}
Combining \eqref{2.1}, \eqref{2.2} and \eqref{2.4}, we arrive at
\begin{equation}\label{2.5}
\begin{split}
\hspace{-1.5em}\Delta_mu=&|\nabla u|^{m-2}\Delta u+(m-2)|\nabla u|^{m-4}\langle D^2u\nabla u,\nabla u\rangle\\
=&|\beta v^{\beta-1}\nabla v|^{m-2}\left[\beta v^{\beta-1}\Delta v+\beta(\beta-1)v^{\beta-2}|\nabla v|^2\right]\\
&+(m-2)|\beta v^{\beta-1}\nabla v|^{m-4}\left[\beta^3v^{3\beta-3}\langle D^2v\nabla v,\nabla v\rangle+\beta^3(\beta-1)v^{3\beta-4}|\nabla v|^4\right]\\
=&|\beta|^{m-2}\beta v^{(\beta-1)(m-1)}|\nabla v|^{m-2}\Delta v\\
&+(m-2)|\beta|^{m-2}\beta v^{(\beta-1)(m-1)}|\nabla v|^{m-4}\langle D^2v\nabla v,\nabla v\rangle\\
&+(m-1)(\beta-1)|\beta|^{m-2}\beta v^{(\beta-1)(m-1)-1}|\nabla v|^m.
\end{split}
\end{equation}
Substituting \eqref{2.1} and \eqref{2.5} into \eqref{1.1}, we have
\begin{equation*}
\begin{split}
&|\beta|^{m-2}\beta v^{(\beta-1)(m-1)}|\nabla v|^{m-2}\Delta v\\
&+(m-2)|\beta|^{m-2}\beta v^{(\beta-1)(m-1)}|\nabla v|^{m-4}\langle D^2v\nabla v,\nabla v\rangle\\
&+(m-1)(\beta-1)|\beta|^{m-2}\beta v^{(\beta-1)(m-1)-1}|\nabla v|^m\;
  =\; - |\beta|^pv^{\beta q+(\beta-1)p}|\nabla v|^p.
\end{split}
\end{equation*}
Dividing by $|\beta|^{m-2}\beta v^{(\beta-1)(m-1)}|\nabla v|^{m-2}$, we derive
\begin{equation}\label{2.6}
\begin{split}
\Delta v=&-(m-1)(\beta-1)\frac{|\nabla v|^2}{v}-(m-2)\frac{\langle D^2v\nabla v, \nabla v\rangle}{|\nabla v|^2}\\
& -|\beta|^{p-m}\beta v^{m-p-1+\beta(p+q-m+1)}|\nabla v|^{p-m+2}.
\end{split}
\end{equation}
Denoting
$$z=|\nabla v|^2,\ \ \ s=m-p-1+\beta(p+q-m+1),$$
we have
\begin{equation}\label{22.6}
\nabla z=2D^2v\nabla v.
\end{equation}
It follows from \eqref{2.6} and \eqref{22.6} that
\begin{equation}\label{3.1}
\Delta v=- (m-1)(\beta-1)\frac{z}{v}-\frac{m-2}{2}\frac{\langle\nabla z,\nabla v\rangle}{z}- |\beta|^{p-m}\beta v^sz^{\frac{p-m+2}{2}}.
\end{equation}
Direct computations imply that
\begin{equation}\label{2.9}
\begin{split}
& \hspace{-2em}\left\langle\nabla\left(\frac{\langle\nabla z,\nabla v\rangle}{z}\right),\nabla v\right\rangle\\
=&\frac{\langle\nabla\langle\nabla z,\nabla v\rangle,\nabla v\rangle}{z}+\langle\nabla z,\nabla v\rangle\left\langle\nabla\left(\frac{1}{z}\right),\nabla v\right\rangle\\
=&\frac{\langle D^2z\nabla v,\nabla v\rangle+\langle D^2v\nabla z,\nabla v\rangle}{z}-\frac{\langle\nabla z,\nabla v\rangle}{z^2}\langle\nabla z,\nabla v\rangle\\
=&\frac{\langle D^2z\nabla v,\nabla v\rangle}{z}+\frac{|\nabla z|^2}{2z}-\frac{\langle\nabla z,\nabla v\rangle^2}{z^2},
\end{split}
\end{equation}
\begin{equation}\label{2.10}
\left\langle\nabla\left(\frac{z}{v}\right),\nabla v\right\rangle=\left\langle-\frac{z\nabla v}{v^2}, \nabla v\right\rangle+\left\langle\frac{\nabla z}{v},\nabla v\right\rangle=-\frac{z^2}{v^2}+\frac{\langle\nabla z, \nabla v\rangle}{v}
\end{equation}
and
\begin{equation}\label{2.11}
\left\langle\nabla\left(v^{s}z^{\frac{p-m+2}{2}}\right),\nabla v\right\rangle=sv^{s-1}z^{\frac{p-m+4}{2}}+\frac{p-m+2}{2}v^{s}z^{\frac{p-m}{2}}\langle\nabla z,\nabla v\rangle.
\end{equation}
Substituting the expressions in \eqref{3.1}-\eqref{2.11} into
the following, we derive
\begin{equation}\label{2.12}
\begin{split}
& \hspace{-2em}\langle\nabla(\Delta v), \nabla v\rangle\\
=&-(m-1)(\beta-1)\left[-\frac{z^2}{v^2}+\frac{\langle\nabla z, \nabla v\rangle}{v}\right] \\
& -\frac{m-2}{2}\left[\frac{\langle D^2z\nabla v,\nabla v\rangle}{z}+\frac{|\nabla z|^2}{2z}-\frac{\langle\nabla z,\nabla v\rangle^2}{z^2}\right]\\
&-|\beta|^{p-m}\beta\left[sv^{s-1}z^{\frac{p-m+4}{2}}+\frac{p-m+2}{2}v^{s}z^{\frac{p-m}{2}}\langle\nabla z,\nabla v\rangle\right]\\
=&(m-1)(\beta-1)\frac{z^2}{v^2}-s|\beta|^{p-m}\beta v^{s-1}z^{\frac{p-m+4}{2}}-(m-1)(\beta-1)\frac{\langle\nabla z,\nabla v\rangle}{v}\\
&-\frac{p-m+2}{2}|\beta|^{p-m}\beta v^sz^{\frac{p-m}{2}}\langle\nabla z,\nabla v\rangle-\frac{m-2}{2}\frac{\langle D^2z\nabla v, \nabla v\rangle}{z}\\
&-\frac{m-2}{4}\frac{|\nabla z|^2}{z}+\frac{m-2}{2}\frac{\langle\nabla z,\nabla v\rangle^2}{z^2}.
\end{split}
\end{equation}
By the Cauchy-Schawarz inequality and \eqref{3.1}, we have
\begin{equation*}
\begin{split}
|D^2v|^2\geq&\frac{1}{N}(\Delta v)^2\\
=&\frac{1}{N}\Big[(m-1)^2(\beta-1)^2\frac{z^2}{v^2}+\frac{(m-2)^2}{4}\frac{\langle\nabla z,\nabla v\rangle^2}{z^2}+\beta^{2(p-m+1)}v^{2s}z^{p-m+2}\\
&+(m-1)(m-2)(\beta-1)\frac{\langle\nabla z,\nabla v\rangle}{v}+2(m-1)(\beta-1)|\beta|^{p-m}\beta v^{s-1}z^{\frac{p-m+4}{2}}\\
&+(m-2)|\beta|^{p-m}\beta v^sz^{\frac{p-m}{2}}\langle\nabla z,\nabla v\rangle\Big].
\end{split}
\end{equation*}
We recall the Bochner formula
\begin{equation*}
\frac12\Delta\left(|\nabla v|^2\right)=|D^2v|^2+\langle\nabla(\Delta v), \nabla v\rangle,\ \ \mathrm{where}\ |D^2v|^2=\sum_{i,j}(v_{ij})^2.
\end{equation*}
So that $z$ satisfies
\begin{equation}\label{3.2}
\begin{split}
&-\frac12\Delta z-\frac{m-2}{2}\frac{\langle D^2z\nabla v, \nabla v\rangle}{|\nabla v|^2}+\left[\frac{(m-1)^2(\beta-1)^2}{N}+(m-1)(\beta-1)\right]\frac{z^2}{v^2}\\
&+\frac{1}{N}\beta^{2(p-m+1)}v^{2s}z^{p-m+2}+\left[\frac{2(m-1)(\beta-1)}{N}-s\right]|\beta|^{p-m}\beta v^{s-1}z^{\frac{p-m+4}{2}}\\
&+(m-1)\left(\frac{m-2}{N}-1\right)(\beta-1)\frac{\langle\nabla z,\nabla v\rangle}{v}+\left[\frac{m-2}{2}+\frac{(m-2)^2}{4N}\right]\frac{\langle\nabla z,\nabla v\rangle^2}{z^2}\\
&+\left[\frac{m-2}{N}-\frac{p-m+2}{2}\right]|\beta|^{p-m}\beta v^sz^{\frac{p-m}{2}}\langle\nabla z,\nabla v\rangle-\frac{m-2}{4}\frac{|\nabla z|^2}{z}\leq0.
\end{split}
\end{equation}
 The coefficients of each term of \eqref{3.2} can be noted as
\begin{equation}\label{2.14a}
\begin{split}
&-\frac12\Delta z-\frac{m-2}{2}\frac{\langle D^2z\nabla v, \nabla v\rangle}{|\nabla v|^2}+A_1\frac{z^2}{v^2} +A_2v^{2s}z^{p-m+2}+A_3 v^{s-1}z^{\frac{p-m+4}{2}}\\
&+A_4\frac{\langle\nabla z,\nabla v\rangle}{v}
+A_5 \frac{\langle\nabla z,\nabla v\rangle^2}{z^2}
+A_6 v^sz^{\frac{p-m}{2}}\langle\nabla z,\nabla v\rangle
+A_7\frac{|\nabla z|^2}{z}\leq0,
\end{split}
\end{equation}
where $A_1$, $A_2$ and $A_3$ shall be crucial in Step 3.

\textbf{Step\ 2}. Estimates on $z$.

For any $\epsilon>0$, the Young's inequality implies that
\begin{equation}\label{2.15a}
\left|A_4\frac{\langle\nabla z, \nabla v\rangle}{v}\right|
\leq |A_4|\frac{|\nabla z||\nabla v|}{v}
=\frac{z}{v}\cdot |A_4|\frac{|\nabla z|}{z^{\frac12}}
\leq\epsilon\frac{z^2}{v^2} + \frac{A_4^2}{4\epsilon}\frac{|\nabla z|^2}{z},
\end{equation}
\begin{equation}\label{2.16a}
\left|A_5\frac{\langle\nabla z, \nabla v\rangle^2}{z^2}\right|
\leq |A_5|\frac{|\nabla z|^2|\nabla v|^2}{z^2}=|A_5|\frac{|\nabla z|^2}{z}
\end{equation}
and
\begin{equation}\label{2.17a}
\left|A_6v^{s}z^{\frac{p-m}{2}}\langle\nabla z, \nabla v\rangle\right|
\leq v^{s}z^{\frac{p-m+2}{2}}\cdot |A_6|\frac{|\nabla z|}{z^{\frac12}}
\leq\epsilon v^{2s}z^{p-m+2}+\frac{A_6^2}{4\epsilon}\frac{|\nabla z|^2}{z}.
\end{equation}
Substituting \eqref{2.15a}-\eqref{2.17a} into \eqref{2.14a}, we infer
\begin{equation}\label{2.18a}
\begin{split}
&-\frac12\Delta z-\frac{m-2}{2}\frac{\langle D^2z\nabla v, \nabla v\rangle}{|\nabla v|^2}
+(A_1-\epsilon)\frac{z^2}{v^2}\\
& + (A_2-\epsilon)v^{2s}z^{p-m+2} + A_3v^{s-1}z^{\frac{p-m+4}{2}}\leq C(\epsilon)\frac{|\nabla z|^2}{z},
\end{split}
\end{equation}
where $C(\epsilon)=\frac{A_4^2}{4\epsilon}+|A_5|+\frac{A_6^2}{4\epsilon}+|A_7|>0$.
We set
$$t \triangleq v^{s+1}z^{\frac{p-m}{2}}.$$
Then
\begin{equation*}
\begin{split}
H &\triangleq (A_1-\epsilon)\frac{z^2}{v^2} + (A_2-\epsilon)v^{2s}z^{p-m+2} + A_3v^{s-1}z^{\frac{p-m+4}{2}} = \frac{z^2}{v^2} T_{\epsilon}(t),
\end{split}
\end{equation*}
where
$$T_{\epsilon}(t)=(A_2-\epsilon)t^2+A_3t+(A_1-\epsilon),$$
which is a quadratic function with respect to $t$. If the discriminant of $T_{\epsilon}(t)$ is negative, there exists $\eta=\eta(N,m,p,q,\beta,\epsilon)>0$ such that
\[ T_{\epsilon}(t)\geq\eta(t^2+1) \geq \eta\max(t^2,1)\geq   \eta t^\theta \m\text{for any }\;
0\leq \theta\leq 2. \]
Since  $\beta>\max(0,\frac{p-m+1}{p+q-m+1})$, we have $s>0$, so that
$\theta=\frac2{s+1}\in (0,2)$. Substituting into the expression of $H$,  we find
\[
H \geq  \frac{z^2}{v^2} \; \eta \left[v^{s+1}z^{\frac{p-m}2}\right]^\theta = \eta z^{\frac{2s+p-m+2}{s+1}}.
\]
It follows from \eqref{2.18a} that
\begin{equation}\label{2.19a}
-\Delta z-(m-2)\frac{\langle D^2z\nabla v, \nabla v\rangle}{|\nabla v|^2}+2\eta z^{\frac{2s+p-m+2}{s+1}}\leq2C(\epsilon)\frac{|\nabla z|^2}{z},
\end{equation}
where
\begin{equation*}
\frac{2s+p-m+2}{s+1}-1=\frac{\beta(p+q-m+1)}{m-p+\beta(p+q-m+1)}>0
\end{equation*}
under our assumptions. 
Applying Lemma \ref{lem3.1} to \eqref{2.19a}, we deduce
\begin{equation*}
z(0)\leq CR^{-2-\frac{2(m-p)}{\beta(p+q-m+1)}},\ \ \mathrm{i.e.,}\ \ \left|\nabla u^{\frac 1{\beta}}(0)\right|\leq CR^{-1-\frac{m-p}{\beta(p+q-m+1)}},
\end{equation*}
where $C=C(N,m,p,q)$. Hence, the conclusion follows with $\alpha=\frac 1{\beta}>0$.

\textbf{Step\ 3.} Study of the quadratic polynomial $T_{\epsilon}(t)$.

Since $T_{\epsilon}$ is a continuous function with respect to $\epsilon$, the negativity of the discriminant of $T_0$ is equivalent to the negativity of the discriminant of $T_{\epsilon}$ for $0<\e<\epsilon_0$ and $\e_0$ small enough.
Clearly,
\begin{equation*}
\begin{split}
T_0(t)=&A_2t^2+A_3t+A_1\\
=&\frac{\beta^{2(p-m+1)}}{N}t^2
+\left[\frac{2(m-1)(\beta-1)}{N}-s\right]|\beta|^{p-m}\beta t\\
&+\frac{(m-1)^2(\beta-1)^2}{N}+(m-1)(\beta-1),
\end{split}
\end{equation*}
its discriminant $D$ satisfies
\begin{equation*}
    \begin{split}
        \beta^{2(m-1-p)}D&=\left[\frac{2(m-1)(\beta-1)}{N}-s\right]^2
        -\frac4N \left[\frac{(m-1)^2(\beta-1)^2}{N}+(m-1)(\beta-1)\right]\\
        &= s^2-\frac{4(m-1)(\beta-1)}{N}s-\frac{4(m-1)(\beta-1)}{N}.
    \end{split}
\end{equation*}
Set $Q=p+q-m+1$. It follows from $\beta-1=\frac{s-q}{p+q-m+1}$ that
\begin{equation*}
   \begin{split}
       D_1(s) &\triangleq NQ\beta^{2(m-1-p)}D\\
       & = NQs^2-4(m-1)(s-q)s-4(m-1)(s-q)\\
       & = [NQ-4(m-1)]s^2 + 4(m-1)(q-1)s + 4(m-1)q,
      \end{split}
\end{equation*}
which is a quadratic polynomial with respect to $s$. Note that $D<0$ is equivalent to $D_1<0$.  According to the sign of the coefficient of $s^2$, we split the discussion into three cases.

\textbf{$\mathrm{Case\ I}$.} The case $NQ-4(m-1)<0$, i.e., $Q<\frac{4(m-1)}{N}$. Since $D_1(s)$ is concave, we would choose large $s$ (large $\beta$) such that $D_1<0$.

\textbf{$\mathrm{Case\ II}$.} The case $NQ-4(m-1)=0$, i.e., $Q=\frac{4(m-1)}{N}$. It is clear that
\begin{equation*}
D_1(s)=4(m-1)(q-1)s+4(m-1)q.
\end{equation*}
If $0\leq q<1$, $D_1(s)$ is decreasing with respect to $s$, we would choose $s>\frac{q}{1-q}$ (i.e., $\beta>\frac{Q-q(p-m+1)}{(1-q)Q}$) such that $D_1<0$.

\textbf{$\mathrm{Case\ III}$.} The case $NQ-4(m-1)>0$, i.e., $Q>\frac{4(m-1)}{N}$. In this case, if $q<1$,
\[
\min_{0<s<\infty}D_1(s)  =  D_1\Big(\frac{2(m-1)(1-q)}{NQ-4(m-1)}\Big)
  =  \frac{4q(m-1)}{NQ-4(m-1)} \Big[ NQ- \frac{(m-1)(1+q)^2}{q}\Big] <0\]
if $NQ<\frac{(m-1)(1+q)^2}{q}$.
\end{proof}

\begin{cor}\label{cor3.1}
Let $\Omega\subset\mathbb R^N$ be a smooth domain. Under the assumptions of Theorem \ref{thm:1.2}, if $u$ is a positive weak solution of \eqref{1.1} in $\Omega$, there exist positive constants $d_0=d_0(\Omega)$ and $C=C(N,m,p,q,\Omega)$ such that
\begin{equation*}
u(x)\leq C\left[\left(\mathrm{dist}(x,\partial\Omega)\right)^{-\frac{m-p}{p+q-m+1}}+\max\{u(z): \mathrm{dist}(z,\partial\Omega)=d_0\}\right]
\end{equation*}
for any $z\in\Omega$. Clearly, the phenomenon of boundary blowup could occur for the case $p<m$.
\end{cor}

\begin{rem}
$\mathrm{(i)}$ Lemma \ref{lem3.2} is considered on $B_R$, which can be viewed as the region after taking a cut-off function in Bernstein-type estimates, so a cut-off function is not introduced in the calculation process.\\
$\mathrm{(ii)}$ The value of $\alpha$ is related to the parameters of transformation in the calculation process, that is, $\alpha=\frac1{\beta}$, which is not easily expressed by $N,m,p,q$. However, this difficulty can be bypassed in the proof of Liouville-type theorem.
\end{rem}

\begin{proof}[Proof of Theorem \ref{thm:1.2}]
When $x\in\mathbb{R}^N$, $R$ which appears in Lemma \ref{lem3.2} can be chosen arbitrarily large. Lemma \ref{lem3.2} implies that $\left|\nabla \left(u^{\alpha}\right)\right|\to0$ as $R\to\infty$. In other words, $u^{\alpha}$ is a constant. It follows from $\alpha>0$ that $u$ is a constant.
\end{proof}

\subsection{The case $q(N-m)+p(N-1)<N(m-1)$}
\begin{proof}[Proof of Theorem \ref{thm:1.4}]
Assume $u$ is a solution of \eqref{s}. For $p+q-m+1\neq0$, set
$$u=v^b,\ \   b(b-1)>0.$$
A calculation which is similar to \eqref{2.6} yields
\begin{equation}\label{s1}
-b\Delta_mv\geq(m-1)(b-1)b\frac{|\nabla v|^m}{v}+|b|^{p-m+2}v^s|\nabla v|^{p}
\end{equation}
with
\begin{equation}\label{ss1}
s=m-p-1+b(p+q-m+1).
\end{equation}
If $s>0$, by the H$\mathrm{\ddot{o}}$lder inequality, we deduce
\begin{equation}\label{s2}
|\nabla v|^{\frac{sm+p}{s+1}}=\left[\frac{|\nabla v|^m}{v}\right]^{\frac{s}{s+1}}\left[v^s|\nabla v|^{p}\right]^{\frac{1}{s+1}}\leq\epsilon^{\frac{s+1}{s}}\frac{|\nabla v|^m}{v}+\epsilon^{-1-s}v^s|\nabla v|^{p}
\end{equation}
for any $\epsilon>0$. From \eqref{s1} and \eqref{s2}, there exists $C>0$ such that
\begin{equation}\label{s3}
-b\Delta_mv\geq C|\nabla v|^{\alpha},
\end{equation}
where
\begin{equation*}
\alpha=\frac{sm+p}{s+1}=\frac{(m-1)(m-p)+mb(p+q-m+1)}{m-p+b(p+q-m+1)}.
\end{equation*}
We claim that
\begin{equation}\label{s4}
\alpha<\frac{N(m-1)}{N-1}.
\end{equation}
We shall discuss three cases as follows.

(i) If $p+q-m+1>0$, we take $b=1+\epsilon$ with $\epsilon>0$ so that
\begin{equation}\label{s5}
s=q+\epsilon(p+q-m+1)
\end{equation}
and $s>m-p-1$. Therefore $\alpha>m-1$. By \eqref{s0}, we choose $\epsilon>0$ small enough such that
\begin{equation}\label{s6}
s(N-m)+p(N-1)<N(m-1),
\end{equation}
which is equivalent to \eqref{s4}.

(ii) If $p+q-m+1<0$, then $m-p-1>q\geq0$. We take $b=-\epsilon$ with $\epsilon>0$, the inequality \eqref{s3} would be written as
\begin{equation*}
-|b|\Delta_mv+C|\nabla v|^{\alpha}\leq0,
\end{equation*}
where $s=m-p-1-\epsilon(p+q-m+1)>m-p-1>0$. Hence
\[
\begin{split}
    \alpha &=\frac{(m-1)[m-p+b(p+q-m+1)]+b(p+q-m+1)}{m-p+b(p+q-m+1)}\\
    &=m-1+\frac{b(p+q-m+1)}{m-p+b(p+q-m+1)}>m-1.
\end{split}
\]
We choose $\epsilon$ small enough such that $\alpha<\frac{N(m-1)}{N-1}.$

(iii) If $p+q-m+1=0$, setting $u=e^{v}$, we derive
\begin{equation*}
-\Delta_mv\geq(m-1)|\nabla v|^m+e^{(p+q-m+1)v}|\nabla v|^{p}.
\end{equation*}
For any $\tilde{\alpha}\in[p,m]$, there exists $\tilde{C}>0$ such that
\begin{equation*}
-\Delta_mv\geq\tilde{C}|\nabla v|^{\tilde{\alpha}}.
\end{equation*}
Under our assumption $N>m$, we may take $m-1<\tilde{\alpha}<\frac{N(m-1)}{N-1}$ such that \eqref{s4} holds.

Set $R>0$. We select a cut-off function $\xi\in C_0^{\infty}(R^N)$, which satisfies
\begin{equation*}
\xi=\left\{\begin{array}{ll}
1,&\mathrm{on}\ B_{\frac{R}{2}},\\
0, &\mathrm{on}\ B_{R}^{c}
\end{array}\right.
\end{equation*}
and
\begin{equation*}
|\nabla\xi|\leq2R^{-1}.
\end{equation*}
In all cases, by multiplying  \eqref{s3} with
$\xi^{\frac{\alpha}{\alpha-m+1}} $ and then integrating by parts, using also the H$\ddot{\mathrm{o}}$lder inequality, we obtain
\begin{equation*}
\begin{split}
\int_{B_R}\xi^{\frac{\alpha}{\alpha-m+1}}|\nabla v|^{\alpha}&\leq C\left|\int_{B_R}|\nabla v|^{m-2}\nabla v\cdot\nabla\left(\xi^{\frac{\alpha}{\alpha-m+1}}\right)\right|\\
&\leq C\left|\int_{B_R}|\nabla v|^{m-1}\xi^{\frac{m-1}{\alpha-m+1}}|\nabla\xi|\right|\\
&\leq\frac12\int_{B_R}\left[\xi^{\frac{m-1}{\alpha-m+1}}|\nabla v|^{m-1}\right]^{\frac{\alpha}{m-1}}+C\int_{B_R}|\nabla\xi|^{\frac{\alpha}{\alpha-m+1}}\\
&=\frac12\int_{B_R}\xi^{\frac{\alpha}{\alpha-m+1}}|\nabla v|^{\alpha}+C\int_{B_R}|\nabla\xi|^{\frac{\alpha}{\alpha-m+1}}.
\end{split}
\end{equation*}
Thus,
\begin{equation*}
\int_{B_R}\xi^{\frac{\alpha}{\alpha-m+1}}|\nabla v|^{\alpha}\leq2C\int_{B_R}|\nabla\xi|^{\frac{\alpha}{\alpha-m+1}}\leq C'R^{N-\frac{\alpha}{\alpha-m+1}};
\end{equation*}
the exponent is negative due to \eqref{s4}, and this
implies the conclusion.
\end{proof}

\section{Applications}

\subsection{Harnack inequalities}

\paragraph{\ \ \ In this subsection, we present some new Harnack inequalities given by Theorem \ref{thm:5.3}. We first give the following useful lemma which will be needed for establishing Theorem \ref{thm:5.3}.}
\begin{lem}\label{lem5.1}
Let $u$ be a positive weak $C^1$ solution of the inequality
\begin{equation}\label{5.1}
-\Delta_mu\geq  bu^q|\nabla u|^p+u^{\alpha_1}-M_2|\nabla u|^{\alpha_2}\ \ \ \mathrm{in}\ \Omega,
\end{equation}
where the constants satisfy $b\ge 0$,  $p,q\in(-\infty,\infty)$, $\alpha_1>m-1$ and $m-1<\alpha_2<\frac{m\alpha_1}{\alpha_1+1}$. Assume   $\gamma\in(0,\alpha_1]\cup \left(0, \frac{N(m-1)}{N-m}\right)$,
$\kappa\in\left[\alpha_1-m+1,\rule{0ex}{.9em}\alpha_1\right]$,
$\mu\in\left(0,\frac{m\alpha_1}{\alpha_1+1}\right)\cup\left(0, \frac{N(m-1)}{N-1}\right)$, $\eta\in (1,m]$.   Let $R_0>0$ and $0<R<R_0$ such that $B_{2R}\subset\Omega$, where $B_{2R}=\left\{x\in\Omega:|x|<2R\right\}$. Then there exists $C=C(N,M_2, m,p,q,\gamma,\mu,\alpha_1,\alpha_2,R_0)>0$, independent of $b$, such that
\begin{eqnarray}\label{5.2}
&\dis\int_{B_R}u^{\gamma}\leq CR^{N-\frac{m\gamma}{\alpha_1-m+1}},  \\
\label{5.11}
&\dis b\int_{B_R}u^{q+\kappa-\alpha_1}|\nabla u|^p\leq CR^{N-\frac{m\kappa}{\alpha_1-m+1}},\\
\label{dd}
&\dis\int_{B_R}u^{-\eta}|\nabla u|^m\leq CR^{N-\frac{m(\alpha_1-\eta+1)}{\alpha_1-m+1}},\\
\label{5.3}
&\dis\int_{B_R}|\nabla u|^{\mu}\leq CR^{N-\frac{(\alpha_1+1)\mu}{\alpha_1-m+1}}.
\end{eqnarray}
\end{lem}
\begin{proof}
Let $\xi(x)$ be a cut-off function on $B_{2}(0)$, which satisfies
\begin{equation*}
0\leq\xi(x)\leq1\ \mathrm{for}\ |x|<2,\ \xi(x)=1\ \mathrm{for}\ |x|\leq 1,\
\xi(x)= 0 \text{ for } |x|\ge 2, \; |\nabla\xi(x)|\leq2.
\end{equation*}
We take
\begin{equation}\label{0.7777}
\phi=\left[\xi\left(\frac{x}{R}\right)\right]^ku^{-d}
\end{equation}
as a test function for \eqref{5.1},
Simple calculations imply that
\begin{equation*}
\nabla\phi=u^{-d}\nabla(\xi^k)-d\xi^ku^{-d-1}\nabla u
\end{equation*}
and
\begin{equation}\label{5.4}
|\nabla\xi^k|=k\xi^{k-1}|\nabla\xi|\leq\frac{2k\xi^{k-1}}{R}.
\end{equation}
It follows from \eqref{5.1} that
\begin{equation}\label{0.6666}
\int_{\Omega}|\nabla u|^{m-2}\nabla u\cdot\nabla\phi\geq b\int_{\Omega}u^q|\nabla u|^p\phi+\int_{\Omega}u^{\alpha_1}\phi-M_2\int_{\Omega}|\nabla u|^{\alpha_2}\phi.
\end{equation}
By \eqref{0.7777}, the inequality \eqref{0.6666} implies
\begin{equation}\label{5.5}
\begin{split}
&d\int_{\Omega}\xi^ku^{-d-1}|\nabla u|^m+ b\int_{\Omega}\xi^ku^{q-d}|\nabla u|^{p}+\int_{\Omega}\xi^ku^{\alpha_1-d}\\
\leq&\int_{\Omega}u^{-d}|\nabla u|^{m-1}|\nabla\xi^k|+M_2\int_{\Omega}\xi^ku^{-d}|\nabla u|^{\alpha_2}.
\end{split}
\end{equation}
{\bf Case 1. $d>0$}.
Combining \eqref{5.4} and the Young's inequality in the form of
\begin{equation}\label{0.11111}
ab\leq\epsilon a^{\beta}+\epsilon^{\frac{1}{1-\beta}}b^{\frac{\beta}{\beta-1}}\ \ \ \forall\beta>1,\ \epsilon>0,
\end{equation}
we obtain,
\begin{equation*}
\int_{\Omega}u^{-d}|\nabla u|^{m-1}|\nabla\xi^k|\leq\frac{d}{2}\int_{\Omega}\xi^ku^{-d-1}|\nabla u|^{m}+CR^{-m}\int_{\Omega}\xi^{k-m}u^{-d-1+m},
\end{equation*}
where $C=C(d)>0$. Thus the inequality \eqref{5.5} implies
\begin{equation}\label{5.6}
\begin{split}
&\frac{d}{2}\int_{\Omega}\xi^ku^{-d-1}|\nabla u|^m+b\int_{\Omega}\xi^ku^{q-d}|\nabla u|^{p}+\int_{\Omega}\xi^ku^{\alpha_1-d}\\
\leq&CR^{-m}\int_{\Omega}\xi^{k-m}u^{-d-1+m}+M_2\int_{\Omega}\xi^ku^{-d}|\nabla u|^{\alpha_2}.
\end{split}
\end{equation}
In the following, we start to estimate each term appearing on the right side of \eqref{5.6}.

First, we focus on the last term on the right side of \eqref{5.6}. By the Young's inequality, we get
\begin{equation*}
u^{-d}|\nabla u|^{\alpha_2}\leq\frac{d}{4M_2}u^{-d-1}|\nabla u|^m+Cu^{\tau},
\end{equation*}
where $\tau=\left[-d+\frac{(d+1)\alpha_2}{m}\right]\frac{m}{m-\alpha_2}$ and $C=C(d,M_2)>0$. So we have
\begin{equation}\label{0.8888}
M_2\int_{\Omega}\xi^ku^{-d}|\nabla u|^{\alpha_2}\leq\frac{d}{4}\int_{\Omega}\xi^ku^{-d-1}|\nabla u|^m+C\int_{\Omega}\xi^ku^{\tau}.
\end{equation}
Since we have assumed $\alpha_2<\frac{m\alpha_1}{\alpha_1+1} $, we derive $\tau<\alpha_1-d$. If we further require
\be \label{d}
d<\frac{\alpha_2}{m-\alpha_2},
\ee
then $0< \tau<\alpha_1-d$. It follows from Young's inequality   that
\begin{equation}\label{0.9999}
C\int_{\Omega}\xi^ku^{\tau}\leq\frac{1}{2}\int_{\Omega}\xi^ku^{\alpha_1-d}+CR^{N},
\end{equation}
where $C>0$. By \eqref{0.8888} and \eqref{0.9999}, we get
\begin{equation}\label{5.7}
M_2\int_{\Omega}\xi^ku^{-d}|\nabla u|^{\alpha_2}\leq\frac{d}{4}\int_{\Omega}\xi^ku^{-d-1}|\nabla u|^m+\frac{1}{2}\int_{\Omega}\xi^ku^{\alpha_1-d}+CR^{N}.
\end{equation}
Substituting into \re{5.6} we obtain
\be \label{g5}
\begin{split}
& \hspace{-2em}\frac{d}{4}\int_{\Omega}\xi^ku^{-d-1}|\nabla u|^m+b\int_{\Omega}\xi^ku^{q-d}|\nabla u|^{p}+\frac 12\int_{\Omega}\xi^ku^{\alpha_1-d}\\
\leq&CR^{-m}\int_{\Omega}\xi^{k-m}u^{-d-1+m} +CR^{N}.
\end{split}
\ee
Since   $\alpha_1>m-1$, we derive $\alpha_1-d>-d-1+m$. Choosing $d \leq m-1 $  (since
$m-1 <\frac{\alpha_2}{m-\alpha_2} $, \re{d} is satisfied) and $k=\frac{m(\alpha_1-d)}{\alpha_1-m+1}$, we deduce from   Young's inequality
(when $d<m-1$) that
\be\label{g4}
\begin{split}
    CR^{-m}\int_{\Omega}\xi^{k-m}u^{-d-1+m}&\leq
    \frac{1}{4}\int_{\Omega}\xi^k u^{\alpha_1-d}+C\int_{\Omega}\xi^{k-\frac{m(\alpha_1-d)}{\alpha_1-m+1}}R^{-\frac{m(\alpha_1-d)}{\alpha_1-m+1}}\\
    &\leq \frac{1}{4}\int_{\Omega}\xi^ku^{\alpha_1-d}+CR^{N-\frac{m(\alpha_1-d)}{\alpha_1-m+1}},
\end{split}
\ee
where $C=C(\alpha_1,m)>0$; the above inequality is obviously valid if $d=m-1$. It follows from \eqref{g5} and \eqref{g4} that
\be \label{g6}
\begin{split}
& \hspace{-2em}\frac{d}{4}\int_{\Omega}\xi^ku^{-d-1}|\nabla u|^m+b\int_{\Omega}\xi^ku^{q-d}|\nabla u|^{p}+\frac 14\int_{\Omega}\xi^ku^{\alpha_1-d}\\
\leq& CR^{N-\frac{m(\alpha_1-d)}{\alpha_1-m+1}}+ CR^N
 \leq CR^{N-\frac{m(\alpha_1-d)}{\alpha_1-m+1}} \mm\text{ for any } 0<d\leq m-1,
\end{split}
\ee
where $C$ depends the upper bound of $R$. Choosing $d=\eta-1$ we
obtain \re{dd}.

For any $\gamma\in (0, \alpha_1)$,
we choose $d$ sufficiently small so that $\gamma<\alpha_1-d$, then by   H\"older's inequality and \re{g6},
\be\label{g2}
\begin{split}
\int_{B_R}u^{\gamma}&
\leq CR^{N-\frac{N\gamma}{\alpha_1-d}} \left[\int_{B_R}u^{\alpha_1-d}\right]^{\frac{\gamma}{\alpha_1-d}}
\leq CR^{N-\frac{m\gamma}{\alpha_1-m+1}}.
\end{split}
\ee
This establishes \re{5.2} with $\gamma\in(0,\alpha_1)$.

Finally, since $0<\alpha_1-\kappa \leq m-1 $,   we choose $d=\alpha_1-\kappa$ in \re{g6} to obtain \re{5.11}
with $ \kappa\in[ \alpha_1-m+1 ,\alpha_1 )$.

{\bf Case 2. $d=0$.} We now let $d=0$ in \re{5.6}.
Since $m-1<\alpha_1$ and $\alpha_2<\frac{m\alpha_1}{\alpha_1+1}$, we apply \re{g2} and \re{5.3} to obtain
\be\label{D0}
\begin{split}
b\int_{\Omega}\xi^ku^{q}|\nabla u|^{p}+\int_{\Omega}\xi^ku^{\alpha_1}
\leq & CR^{-m}\int_{\Omega}\xi^{k-m}u^{m-1}+M_2\int_{\Omega}\xi^k|\nabla u|^{\alpha_2}\\
\leq & CR^{N-\frac{m\alpha_1}{\alpha_1-m+1}} + CR^{N-\frac{(\alpha_1+1)\alpha_2}{\alpha_1-m+1}}\\
\leq & CR^{N-\frac{m\alpha_1}{\alpha_1-m+1}},
\end{split}
\ee
where $C$ depends the upper bound of $R$. Thus we obtain
\re{5.2} with $\gamma=\alpha_1$ and \re{5.11} with $\kappa = \alpha_1$.

{\bf Case 3.} For $0<\gamma<\frac{N(m-1)}{N-m}$, take $0<d< m-1$
such that
$\gamma=\frac{N(m-d-1)}{N-m}$. Since $0<m-d-1<m-1<\alpha_1$, it follows from the Sobolev embedding theorem, \re{g6} and \re{5.2} that
\be\label{D1a}
\begin{split}
    \left[\int_{\Omega}\left(\xi^{\frac km}u^{\frac{m-d-1}m}\right)^{m^*}\right]^{\frac{m}{m^*}}
    & \leq C\int_{\Omega}\left|\nabla\left(\xi^{\frac km}u^{\frac{m-d-1}m}\right)\right|^m\\
    & \leq C\int_{\Omega}\xi^ku^{-d-1}|\nabla u|^m + C\int_{\Omega}\xi^{k-m}u^{m-d-1}|\nabla \xi|^{m}\\
    & \leq CR^{N-\frac{m(\alpha_1-d)}{\alpha_1-m+1}} + CR^{-m}\int_{\Omega}\xi^{k-m}u^{m-d-1}\\
    & \leq CR^{N-\frac{m(\alpha_1-d)}{\alpha_1-m+1}},
\end{split}
\ee
where $m^*=\frac{Nm}{N-m}$. A computation shows
\[
\left[N-\frac{m(\alpha_1-d)}{\alpha_1-m+1}\right]\frac{m^*}{m}=
N-\frac{mN(m-d-1)}{(\alpha_1-m+1)(N-m)},
\]
so that
\be \label{gg}
\int_{B_R} u^\gamma = \int_{B_R} u^{\frac{N(m-d-1)}{N-m} } \le
CR^{N-\frac{m\gamma}{\alpha_1-m+1}}.
\ee
This completes the proof of all cases of \re{5.2}.

We now proceed to derive \re{5.3}. If $\mu<m$, then
we apply   H\"{o}lder's inequality to derive
\begin{equation*}
\int_{B_R}|\nabla u|^{\mu}\leq\left[\int_{B_R}u^{-d-1}|\nabla u|^{m}\right]^{\frac{\mu}{m}}\left[\int_{B_R}u^{\frac{(d+1)\mu}{m-\mu}}\right]^{1-\frac{\mu}{m}}.
\end{equation*}

{\bf Case  $0<\mu<\frac{m\alpha_1}{\alpha_1+1} $:} In this case $\mu<m$.
The assumption  $\mu<\frac{m\alpha_1}{\alpha_1+1}$ is equivalent to
$\frac{ \mu}{m-\mu}<\alpha_1$. We choose $d$ sufficiently small such that $\frac{(d+1)\mu}{m-\mu}<\alpha_1$. Combining \eqref{g6} and \eqref{g2}, we deduce \eqref{5.3}.

{\bf Case $0<\mu< \frac{N(m-1)}{N-1}$: } Since $N>m$, in this case we also have $\mu<m$, and, $\frac{\mu}{m-\mu}<\frac{N(m-1)}{N-m}$.
We choose $d$ sufficiently small such that $\frac{(d+1)\mu}{m-\mu}<\frac{N(m-1)}{N-m}$. Combining \eqref{g6} and \eqref{gg}, we deduce \eqref{5.3}.
\end{proof}

The following lemma is taken from \cite{R},
which is a generalization of Lemma 4.2 in \cite{SZ}. Both lemmas are associated
with the quasilinear equation
\begin{equation*}
\mathrm{div}\mathcal{A}(x,u,\nabla u)=\mathcal{B}(x,u,\nabla u)
\end{equation*}
considered in \cite[Theorem 5]{S1}  under appropriate assumptions on $\mathcal{A}$ and $\mathcal{B}$. After restriction to the operator $\mathcal{A}(x,u,\nabla u)=|\nabla u|^{m-2}\nabla u$, the factors $R^{1-\frac{N}{\sigma'}}$ and $R^{m-\frac{N}{\sigma}}$ come naturally to rescale to the unit ball $B_1$ in the process of proof.
\begin{lem}\label{lem5.2}(\cite[Lemma 2.2]{R})
Let $u$ be a nonnegative weak solution of inequality
\begin{equation*}
|\Delta_mu|\leq c(x)|\nabla u|^{m-1}+d(x)u^{m-1}+f(x)\ \ \mathrm{in}\ \Omega,
\end{equation*}
where $c\in L^{\sigma'}(\Omega)$, $d,f\in L^{\sigma}(\Omega)$, $\sigma'>N$ and $\sigma\in\left(\frac{N}{m},\frac{N}{m-1}\right)$. Then for every $R>0$ such that $B_{2R}\subset\Omega$, there exists a positive constant $C$ depending on
\begin{equation*}
N,\ m,\ \sigma,\ \sigma',\ R^{1-\frac{N}{\sigma'}}\|c\|_{L^{\sigma'}},\ R^{m-\frac{N}{\sigma}}\|d\|_{L^{\sigma}}
\end{equation*}
such that
\begin{equation*}
\sup_{B_R}u\leq C\left[\inf_{B_R}u+R^{m-\frac{N}{\sigma}}\|f\|_{L^{\sigma}}\right].
\end{equation*}
\end{lem}

Combining Lemmas \ref{lem5.1} and \ref{lem5.2}, we derive the following Harnack inequality that is useful for deriving a priori estimates in Subsection 3.2.

\begin{thm}\label{thm:5.3}
Let $u$ be a positive weak solution of the inequality
\begin{equation*}
bu^q|\nabla u|^p+u^{\alpha_1}\!-M_2|\nabla u|^{\alpha_2}\!\leq - \Delta_mu\leq c_0bu^q|\nabla u|^p\!+ M_1 u^{\alpha_1}\!+M_2|\nabla u|^{\alpha_2}\!+\lambda\ \ \mathrm{in}\ \Omega,
\end{equation*}
where $b=1$ or $0$, and
the constants satisfy $\lambda>0$, $c_0,M_1\geq1$, $1<m<N$, $m-1<p<m$, $q\ge 0$, $\frac{p+mq}{m-p}\leq\alpha_1<\frac{N(m-1)}{N-m}$, $m-1<\alpha_2<\frac{m\alpha_1}{\alpha_1+1}$, and $q(N-m)+p(N-1)<N(m-1)$.  Let $R_0>0$ and $0<R<R_0$ such that $B_{2R}\subset\Omega$. Then there exists a positive constant $C$ depending on $N,m,p,q,\alpha_1,\alpha_2,c_0, c_1,M_2,R_0$ such that
\begin{equation*}
\sup_{B_R}u\leq C\left[\inf_{B_R}u+R^m\lambda\right].
\end{equation*}
\end{thm}
\begin{proof}
We apply Lemma \ref{lem5.2} with
$$f(x)=\lambda,\ \  d(x)= M_1u^{\alpha_1-m+1} $$
and
$$c(x)=M_2|\nabla u|^{\alpha_2-m+1}+c_0bu^q|\nabla u|^{p-m+1}\triangleq c_1(x)+c_2(x).$$
So it suffices to prove that
\begin{equation}\label{5.12}
R^{m-\frac{N}{\sigma}}\|f\|_{L^{\sigma}}<C\lambda R^{m},
\end{equation}
\begin{equation}\label{5.13}
R^{1-\frac{N}{\sigma'}}\|c\|_{L^{\sigma'}}<C,
\end{equation}
and
\begin{equation}\label{5.14}
R^{m-\frac{N}{\sigma}}\|d\|_{L^{\sigma}}<C
\end{equation}
with $C>0$.

The inequality \eqref{5.12} would be obtained by a direct calculation
\begin{equation*}
R^{m-\frac{N}{\sigma}}\|f\|_{L^{\sigma}}=R^{m-\frac{N}{\sigma}}\left[\int_{B_{2R}}\lambda^\sigma\right]^{\frac{1}{\sigma}}=CR^{m-\frac{N}{\sigma}}\lambda(2R)^{\frac{N}{\sigma}}=C\lambda R^m.
\end{equation*}
We now start to establish \eqref{5.13}. It is clear that
\begin{equation}\label{5.15}
R^{1-\frac{N}{\sigma'}}\|c\|_{L^{\sigma'}}=R^{1-\frac{N}{\sigma'}}\|c_1+c_2\|_{L^{\sigma'}}\leq R^{1-\frac{N}{\sigma'}}\|c_1\|_{L^{\sigma'}}+R^{1-\frac{N}{\sigma'}}\|c_2\|_{L^{\sigma'}}.
\end{equation}
Let us estimate each term appearing on the right side of \eqref{5.15}.

First, we focus on the first term on the right side of \eqref{5.15}. A simple calculation implies that
\begin{equation*}
R^{1-\frac{N}{\sigma'}}\|c_1\|_{L^{\sigma'}}=M_2R^{1-\frac{N}{\sigma'}}\left[\int_{B_{2R}}|\nabla u|^{(\alpha_2-m+1)\sigma'}\right]^{\frac{1}{\sigma'}}\leq M_2CR^{1-\frac{(\alpha_1+1)(\alpha_2-m+1)}{\alpha_1-m+1}},
\end{equation*}
the last inequality holds due to \eqref{5.3} if we verify that
\begin{equation}\label{5.16}
0<(\alpha_2-m+1)\sigma'<\frac{m\alpha_1}{\alpha_1+1}.
\end{equation}
Choosing $\sigma'>N$ close enough to $N$, it suffices to show that
$$0<(\alpha_2-m+1)N<\frac{m\alpha_1}{\alpha_1+1}.$$
Since $\alpha_2<\frac{m\alpha_1}{\alpha_1+1}$, we just need to verify
\begin{equation}\label{0.33333}
\left[\frac{m\alpha_1}{\alpha_1+1}-m+1\right]N<\frac{m\alpha_1}{\alpha_1+1},\ \mathrm{i.e.,}\ \alpha_1<\frac{N(m-1)}{N-m} ,  \m  1<m<N.
\end{equation}
The assumptions in \re{F} imply that \eqref{0.33333} holds.
Furthermore,
$$1-\frac{(\alpha_1+1)(\alpha_2-m+1)}{\alpha_1-m+1}>0\ \ \text{is equivalent to our assumption }\ \alpha_2<\frac{m\alpha_1}{\alpha_1+1}.$$
Thus there exists $C>0$ such that
\begin{equation}\label{5.17}
R^{1-\frac{N}{\sigma'}}\|c_1\|_{L^{\sigma'}}\leq M_2CR^{1-\frac{(\alpha_1+1)(\alpha_2-m+1)}{\alpha_1-m+1}}\leq C
\end{equation}
for $R\leq R_0$.

Second, we focus on the second term on the right side of \eqref{5.15}. If $b=0$, then $c_2(x)\equiv 0$ and there is nothing to proof. So we assume that $b=1$. By the H\"{o}lder inequality, \eqref{5.2} and \eqref{5.11} with $\kappa=\alpha_1$,  assuming the conditions of Lemma \ref{lem5.1} are satisfied, we have
\begin{equation}\label{0.666}
\int_{B_{2R}}\left[u^q|\nabla u|^{p-m+1}\right]^{\sigma'}
\leq\left[\int_{B_{2R}}u^{q}|\nabla u|^{p}\right]^{\frac{(p-m+1)\sigma'}{p}}\left[\int_{B_{2R}}u^{\beta}\right]^{1-\frac{(p-m+1)\sigma'}{p}}
\le  CR^\eta,
\end{equation}
where
\begin{equation}\label{T41}
\begin{split}
 & \eta = \Big(N-\frac{m\alpha_1}{\alpha_1-m+1}\Big)\frac{(p-m+1)\sigma'}{p} \\
 & \hspace{3em}+\Big(N-\frac{m\beta}{\alpha_1-m+1}\Big)\Big[1-\frac{(p-m+1)\sigma'}{p}\Big], \\
 & \beta=\frac{q\sigma'(m-1)}{p-(p-m+1)\sigma'}.
\end{split}
\end{equation}
We need to verify the conditions of  Lemma \ref{lem5.1}, that is, $0<\beta< \frac{N(m-1)}{N-m}$, i.e.,
\begin{equation*}
0<\frac{q\sigma'(m-1)}{p-(p-m+1)\sigma'}< \frac{N(m-1)}{N-m}.
\end{equation*}
Choosing $\sigma'>N$ close enough to $N$, we just need to show that
\begin{equation*}
0<\frac{qN(m-1)}{p-(p-m+1)N}\triangleq A<\frac{N(m-1)}{N-m}.
\end{equation*}
It follows from $q(N-m)+p(N-1)<N(m-1)$ and $m<N$ that
\begin{equation}\label{T50}
p(N-1)<N(m-1),\ \mathrm{i.e.,}\ p-(p-m+1)N>0.
\end{equation}
Then $A>0$. On the other hand, by \eqref{T50}, the inequality $A<\frac{N(m-1)}{N-m}$ is equivalent to
$q(N-m)+p(N-1)<N(m-1)$. It follows from \eqref{0.666}, \eqref{T41} that
\begin{equation}\label{Y11}
\begin{split}
R^{1-\frac{N}{\sigma'}}\|c_2\|_{L^{\sigma'}}&=c_0R^{1-\frac{N}{\sigma'}}\left[\int_{B_{2R}}\left(u^q|\nabla u|^{p-m+1}\right)^{\sigma'}\right]^{\frac{1}{\sigma'}}
\le CR^{1-\frac{N}{\sigma'}+\frac{\eta}{\sigma'}}, 
\end{split}
\end{equation}
where
\beaa
 1-\frac{N}{\sigma'}+\frac{\eta}{\sigma'} & = & 1-\frac{m\alpha_1(p-m+1)}{p(\alpha_1-m+1)} - \frac{m\beta}{\sigma'(\alpha_1-m+1)}
 +\frac{m\beta(p-m+1)}{p(\alpha_1-m+1)} \\
 & = & 1- \frac{m\alpha_1\sigma'(p-m+1)+m\beta[p-(p-m+1)\sigma']}{p\sigma'(\alpha_1-m+1)} \\
 & = &
 1-\frac{m\alpha_1(p-m+1)+mq(m-1)}{p(\alpha_1-m+1)}\\
 &= &
 \frac{(m-1)[(m-p)\alpha_1-(p+mq)]}{p(\alpha_1-m+1)}\\
  & \geq & 0  \m\text{since }
\alpha_1\geq\frac{p+mq}{m-p}.
\eeaa
Thus, when $R\leq R_0$, we have
\begin{equation}\label{0.66666}
R^{1-\frac{N}{\sigma'}}\|c_2\|_{L^{\sigma'}}\leq C.
\end{equation}
Combining \eqref{0.66666}, \eqref{5.17} and \eqref{5.15}, we deduce \eqref{5.13}.

Now, we prove \eqref{5.14}. A simple calculation implies that,
if $0<(\alpha_1-m+1)\sigma<\frac{N(m-1)}{N-m}$, then
\begin{equation}\label{5.20}
R^{m-\frac{N}{\sigma}}\|d\|_{L^{\sigma}}=M_1R^{m-\frac{N}{\sigma}}\left[\int_{B_{2R}}u^{(\alpha_1-m+1)\sigma}\right]^{\frac{1}{\sigma}}\leq M_1R^{m-\frac{N}{\sigma}}R^{\frac1\sigma\big(N-\frac{m\sigma(\alpha_1-m+1)}{\alpha_1-m+1}\big)}=M_1,
\end{equation}
we deduce \eqref{5.14}. Now we   verify
\begin{equation}\label{5.21}
0<(\alpha_1-m+1)\sigma< \frac{N(m-1)}{N-m}.
\end{equation}
Choosing $\sigma>\frac{N}{m}$ close enough to $\frac{N}{m}$, it suffices to show that
\begin{equation}\label{0.77777}
(\alpha_1-m+1)\frac{N}{m}<\frac{N(m-1)}{N-m},\   \mathrm{i.e.,}\
\alpha_1<\frac{N(m-1)}{N-m},
\end{equation}
which is already evaluated in \re{0.33333}. The proof is complete.
\end{proof}

\begin{rem}\label{rem5.2}
 Theorem \ref{thm:5.3} is a generalization of Theorem 2.3 in \cite{R}, which considered the  special case $b=0$
The ranges of parameters are consistent with those of this paper. Furthermore, Theorem 2.3 in \cite{R} is a generalization of Theorem 4.1(b) in \cite{SZ}, which considered the inequality
$$u^{p-1}-u^{m-1}-|\nabla u|^{m-1}\leq-\Delta_mu\leq \Lambda\left[u^{s-1}+u^{m-1}+|\nabla u|^{m-1}\right]$$
with $m<p\leq s<m_*$. Clearly, only the case $\alpha=m-1$ in \cite{R} is considered in \cite{SZ}.
\end{rem}

\subsection{A priori estimates}
\paragraph{\ \ \ This subsection is devoted to study the following problem}
\begin{equation}\label{6.1}
\left\{\begin{array}{ll}
-\Delta_mu=f(x,u,\nabla u)+\lambda,&x\in\Omega,\\
u(x)=0,&x\in\partial\Omega,
\end{array}\right.
\end{equation}
where $\lambda>0$ will be used to make a compact homotopy in next subsection. We shall provide a priori estimates by using Harnack inequalities, the blow-up procedure and the Liouville-type theorem.

The following inequality is an easy consequence of the Picone identity for the $m$-Laplacian.  The version stated in  \cite[Lemma 3.1]{R} requires $h(x)\ge 0$, this assumption guarantees that $\phi_1^m/u^{m-1}\in W_0^{1,m}(\Omega)$.
  If ``$h(x)\ge 0$'' is not assumed, then the following version of the estimate holds.
\begin{lem}\label{lem111}
Let $u$ be a positive solution of the problem
\begin{equation*}
\left\{\begin{array}{ll}
-\Delta_mu=h(x), 
&x\in\Omega,\\
u(x)=0,&x\in\partial\Omega.
\end{array}\right.
\end{equation*}
Then
\begin{equation*}
  \sup_{\epsilon>0}\int_{\Omega}h(x)\frac{\phi_1^m}{(u+\epsilon)^{m-1}}\leq\lambda_1\int_{\Omega}\phi_1^m,
\end{equation*}
where $\dis\lambda_1=\inf_{v\in W^{1,m}_0(\Omega)} \int_\Omega |\nabla v|^m \Big/\int_\Omega |v|^m $ is the first eigenvalue of the $m$-Laplacian with Dirichlet boundary conditions, and $\phi_1$ is the associated eigenfunction.
\end{lem}
\begin{proof}
Integrating by parts, we obtain, for any $\epsilon>0$,
 \beaa
 \lefteqn{ \int_{\Omega}h(x)\frac{\phi_1^m}{(u+\epsilon)^{m-1}}  =
 \int_\Omega |\nabla u|^{m-2}\nabla u\cdot \nabla\left(\frac{\phi_1^m}{(u+\epsilon)^{m-1}}\right) }\\
  & = & \int_\Omega -(m-1) (u+\epsilon)^{-m}|\nabla u|^{m} \phi_1^m
  + m (u+\epsilon)^{-(m-1)}\phi_1^{m-1} |\nabla u|^{m-2}\nabla u\cdot \nabla \phi_1.
 \eeaa
Using the Young's inequality $ab\le \frac{m-1}m a^{\frac{m}{m-1}} + \frac1m b^m$, we find that
\beaa
(u+\epsilon)^{-(m-1)}\phi_1^{m-1} |\nabla u|^{m-2}\nabla u\cdot \nabla \phi_1
& \le & \Big((u+\epsilon)^{-(m-1)}\phi_1^{m-1} |\nabla u|^{m-1}\Big)|\nabla \phi_1 |\\
& \le & \frac{m-1}m   (u+\epsilon)^{-m}|\nabla u|^{m} \phi_1^m
 + \frac1m |\nabla \phi_1|^m.
\eeaa
Thus
 \be
 \int_{\Omega}h(x)\frac{\phi_1^m}{(u+\epsilon)^{m-1}} \le \int_\Omega  |\nabla \phi_1|^m
  =   \lambda_1\int_{\Omega}\phi_1^m,
 \ee
 where the last equality comes from the fact that $\phi_1$ is the first eigenfunction. The desired estimate follows immediately.
\end{proof}

The following result shows the nonexistence of positive weak solutions of \eqref{6.1} for $\lambda$ large.
\begin{prop}\label{prop6.1}
Assume that   $f(x,u,\eta)\ge u^{\alpha_1}-M_2|\eta|^{\alpha_2}$, where $M_1, M_2>0$ and the exponents $\alpha_1$, $\alpha_2$ satisfy $\alpha_1>m-1$ and $m-1<\alpha_2<\frac{m\alpha_1}{\alpha_1+1}$. Then there exists $\lambda_0>0$
depending on $m$, $\alpha_1$, $\alpha_2$, $M_1$,  $M_2$ and $\Omega$ such that the problem \eqref{6.1} admits no positive weak solutions for any $\lambda\geq\lambda_0$.
\end{prop}
\begin{proof}
We proceed by contradiction. Assume that $u$ is a positive weak solution of \eqref{6.1}. It follows from Lemma \ref{lem111} that,
for any $0<\epsilon\le 1$,
\begin{equation}\label{m1}
\begin{split}
\lambda_1\int_{\Omega}\phi_1^m&\geq\int_{\Omega}\left[f(x,u,\nabla u)+\lambda\right]\frac{\phi_1^m}{(u+
\epsilon)^{m-1}}\\
&\geq\int_{\Omega}\left[u^{\alpha_1}+\lambda\right]\frac{\phi_1^m}{(u+ \epsilon)^{m-1}}-\left[M_2|\nabla u|^{\alpha_2}\right]\frac{\phi_1^m}{(u+ \epsilon)^{m-1}}.
\end{split}
\end{equation}
Define
\begin{equation*}
 l(\epsilon,
\lambda)=\min\left\{\frac{\lambda+t^{\alpha_1}}{(t+\epsilon)^{m-1}}:t\geq0\right\}.
\end{equation*}

Clearly,
\begin{equation}\label{m2}
\lim_{\lambda\to\infty}l(\epsilon,\lambda)=\infty \m\text{ uniformly for $0\le \epsilon\le 1$}.
\end{equation}
It follows from \eqref{m1} that
\begin{equation*}
\lambda_1\int_{\Omega}\phi_1^m\geq l\int_{\Omega}\phi_1^m-\int_{\Omega}\left[M_2|\nabla u|^{\alpha_2}\right]\frac{\phi_1^m}{(u+\epsilon)^{m-1}},
\end{equation*}
i.e.,
\begin{equation}\label{m3}
(l-\lambda_1)\int_{\Omega}\phi_1^m\leq \int_{\Omega}\left[M_2|\nabla u|^{\alpha_2}\right]\frac{\phi_1^m}{(u+\epsilon)^{m-1}}.
\end{equation}
We claim that there exists $C>0$ independent of $\epsilon$ and $u$ such that
\begin{equation}\label{m10}
\int_{\Omega}\left[M_2|\nabla u|^{\alpha_2}\right]\frac{\phi_1^m}{(u+\epsilon)^{m-1}}\leq \int_{\Omega}\left[M_2|\nabla u|^{\alpha_2}\right]\frac{\phi_1^m}{u^{m-1}}\le C.
\end{equation}
By \eqref{m3} and the boundedness of $\phi_1$, we know that $l$ is bounded, which is a contradiction with \eqref{m2}.

In the following, we start to prove \eqref{m10}. Multiplying \eqref{6.1} by the test function $\psi=\frac{\phi_1^m}{(u+\epsilon)^{m-1}}$ and integrating, we obtain
\begin{equation}\label{m4}
\begin{split}
&\hspace{-2em}(m-1)\int_{\Omega}\phi_1^m(u+\epsilon)^{-m}|\nabla u|^m+\int_{\Omega}\phi_1^mu^{\alpha_1} (u+\epsilon)^{1-m}\\
\leq&m\int_{\Omega}\phi_1^{m-1}(u+\epsilon)^{-(m-1)}|\nabla u|^{m-1}|\nabla\phi_1|+M_2\int_{\Omega}\phi_1^m(u+\epsilon)^{1-m}|\nabla u|^{\alpha_2}.
\end{split}
\end{equation}
By the boundedness of $\phi_1$, $\nabla\phi_1$ and the Young's inequality in the form of
\begin{equation*}
x<\eta x^{\gamma}+C(\eta),\ \gamma>1,\ \eta>0,
\end{equation*}
we derive
\begin{equation}\label{m5}
\begin{split}
m\int_{\Omega}\phi_1^{m-1}(u+\epsilon)^{-(m-1)}|\nabla u|^{m-1}|\nabla\phi_1|\leq&Cm\int_{\Omega}\phi_1^{m-1}(u+\epsilon)^{-(m-1)}|\nabla u|^{m-1}\\
\leq&\frac{m-1}{2}\int_{\Omega}\phi_1^{m}(u+\epsilon)^{-m}|\nabla u|^{m}+C.
\end{split}
\end{equation}
A calculation which is similar to \eqref{5.7} yields
\begin{equation}\label{m8}
\begin{split}
M_2\int_{\Omega}&\phi_1^m(u+\epsilon)^{1-m}|\nabla u|^{\alpha_2} \\
&\leq  \frac{m-1}{4}\int_{\Omega}\phi_1^m(u+\epsilon)^{-m}|\nabla u|^m+  \frac 12\int_{\Omega}\phi_1^m(u+\epsilon)^{\alpha_1-m+1}+C.
\end{split}
\end{equation}
Combining \eqref{m4}-\eqref{m8}, we get
\be\label{m9}
\begin{split}
\frac{m-1}4&\int_{\Omega}\phi_1^m(u+\epsilon)^{-m}|\nabla u|^m+\frac 12\int_{\Omega}\phi_1^mu^{\alpha_1}(u+\epsilon)^{1-m} \\
 & \le C + \frac 12 \int_\Omega \phi_1^m  \left[(u+
 \epsilon)^{\alpha_1-m+1}-u^{\alpha_1}(u+\epsilon)^{-m+1}\right].
 \end{split}
\ee
Letting $\epsilon\to 0+$, we find that the estimates \eqref{m8}-\eqref{m9} implies that the claim \eqref{m10} holds. The proof is complete.
\end{proof}

The following theorem presents a strong maximum principle, which plays a vital role to get a priori estimates.
\begin{thm}\label{thm:6.1}
Let $u\in C^1(\Omega)$ be a nonnegative weak solution of
\begin{equation*}
-\Delta_mu+u^{\alpha}=f\ \ \ \mathrm{in}\ \Omega,
\end{equation*}
where $m>1$ and $f\geq0,\ \mathrm{a.e.\ in}\ \Omega$. If $\alpha>m-1$ and $u$ does not vanish identically, then $u$ is positive everywhere in $\Omega$.
\end{thm}
\begin{proof}
In view of the classical work \cite[(12)-(13$'$), page 194]{V}, it suffices to prove that
\begin{equation*}
\int_0^1\left[u^{\alpha}u\right]^{-\frac{1}{m}} du=\infty.
\end{equation*}
In fact, it follows from $\alpha>m-1$ that $\frac{m-\alpha-1}{m}<0$. Thus,
\begin{equation*}
\begin{split}
\int_0^1\left[u^{\alpha}u\right]^{-\frac{1}{m}}du=\frac{m}{m-\alpha-1}u^{\frac{m-\alpha-1}{m}}\bigg|_{0}^{1}
=\frac{m}{m-\alpha-1}\left[1-\lim_{u\to0^+}u^{\frac{m-\alpha-1}{m}}\right]=\infty.
\end{split}
\end{equation*}
The conclusion holds.
\end{proof}
The following theorem presents a Liouville result on the half-space $\mathbb{R}_+^N$.
\begin{thm}\label{lem:6.1a}(See \cite[Theorem 1.1]{Z})
Assume that $B(u)$ is continuously differentiable for $u>0$ and that there exist positive constants $K>0$, $\gamma_1\in\left(m-1,\frac{N(m-1)+m}{N-m}\right)$ and $\gamma_2\in\left(0,\frac{N(m-1)+m}{N-m}\right)$ such that for $u>0$
\[
K^{-1}u^{\gamma_1}\leq B(u)\leq Ku^{\gamma_1}, \mm \gamma_2B(u)\geq u B'(u).
\]
Then the equation
\[
-\Delta_m u=B(u), \m x\in \mathbb{R}^N_+
\]
does not admit any non-negative non-trivial solutions $u$ on $\mathbb{R}^N_+$ vanishing on $\partial \mathbb{R}^N_+$.
\end{thm}

The following proposition will give a priori estimates about positive weak solutions of \eqref{6.1} when $\lambda$ is bounded.
\begin{prop}\label{prop6.2}
Assume that $f:\bar\Omega\times\mathbb R\times\mathbb R^N\to\mathbb R$ satisfies \re{F}.
If $u\in C^1(\Omega)$ is a positive weak solution of \eqref{6.1}, then there exists $C>0$ such that $\|u\|_{\infty}\le C$, where $\|\cdot\|_{\infty}$ is the uniform supremum norm.
\end{prop}
\begin{proof}
We proceed by contradiction.  Assume that there exist $\lambda_n>0$  such that $u_n$ are positive weak solutions of the problem
\begin{equation}\label{6.2}
\left\{\begin{array}{ll}
-\Delta_mu_n=f(x,u_n,\nabla u_n)+\lambda_n,&x\in\Omega,\\
u_n(x)=0,&x\in\partial\Omega,
\end{array}\right.
\end{equation}
and that $\|u_n\|_{\infty}\to\infty$ as $n\to\infty$. The existence of a positive solution and an application of  Proposition \ref{prop6.1} imply that $\lambda_n \le \lambda_0$.

Motivated by \cite[Proposition 3.3]{R}, we will make the blow-up procedure around a fixed point $y_0$ in $\Omega$. Let $x_n$ be points in $\Omega$ such that $u_n(x_n)=\|u_n\|_{\infty}=S_n$. Denote $\delta_n=\mathrm{dist}(x_n,\partial\Omega)$. In the following, we start to prove that
\begin{equation}\label{0.3}
u_n(y_0)\to\infty\ \mathrm{as}\ n\to\infty\ \mathrm{for\ some}\ y_0\in\Omega.
\end{equation}
The present paper proceeds the discussion in four steps.

\textbf{Step 1.} We claim that there exists $C>0$ such that $S_n^{\frac{\alpha_1-m+1}{m}}\delta_n\ge C.$

The result will be obtained by the blow-up argument around the points $x_n$ in which $u_n$ attain their maxima. Here and in the following, we denote $C$ by positive constants which are independent of $n$. Define
\begin{equation}\label{6.3}
w_n(y)= S_n^{-1}u_n(M_ny+x_n - \delta_n \vec e),
\end{equation}
where $M_n>0$ will be chosen later,  $z_n\in\partial\Omega$ is the projection such that $\dist(x_n, \partial\Omega)=|x_n-z_n|$ and $\dis\vec e = \frac{x_n-z_n}{\dist(x_n, \partial\Omega)}$, $|\vec e|=1$.
Clearly, the functions $w_n$ are well defined in
\[
 \Omega_n = \{ y: M_n y + x_n - \delta_n \vec e\in \Omega\}
\]
and satisfy
\begin{equation}\label{0.74}
w_n(M_n^{-1}\delta_n \vec e)=\|w_n\|_{L^\infty(\Omega_n)}=1,\ \ \
w_n(0)=S_n^{-1}u_n(x_n-\delta_n \vec e)= S_n^{-1} u(z_n)=0,
\end{equation}
  Rotate the system if necessary, we assume without loss of generality that
\be
 \text{$\vec e$ is in the positive $x_N$ direction.}
\ee
By \eqref{6.2} and \eqref{6.3}, we derive
\begin{equation}\label{6.4}
\begin{split}
-\Delta_mw_n(y)=&-S_n^{1-m}M_n^m\Delta_mu_n \\
=&S_n^{1-m}M_n^m\left[f(M_ny+x_n-\delta_n \vec e ,S_nw_n(y),S_nM_n^{-1}\nabla w_n(y))+\lambda_n\right]\\
\triangleq&\theta_n(y,w_n,\nabla w_n).
\end{split}
\end{equation}
It follows from \eqref{F} that
\begin{equation*}
\begin{split}
&\hspace{-1em}|\theta_n(y,w_n,\nabla w_n)|\\
\leq&S_n^{1-m}M_n^m\left[c_0(S_nw_n)^q|S_nM_n^{-1}\nabla w_n|^p+M_1(S_nw_n)^{\alpha_1}+M_2|S_nM_n^{-1}\nabla w_n|^{\alpha_2}+\lambda_n\right]\\
=&c_0S_n^{p+q-m+1}M_n^{m-p}w_n^q|\nabla w_n|^p+M_1S_n^{1-m+\alpha_1}M_n^mw_n^{\alpha_1}+M_2S_n^{1-m+\alpha_2}M_n^{m-\alpha_2}|\nabla w_n|^{\alpha_2}\\
&+\lambda_nS_n^{1-m}M_n^m.
\end{split}
\end{equation*}
We choose $S_n^{1-m+\alpha_1}M_n^m=1$, i.e., $M_n=S_n^{\frac{m-1-\alpha_1}{m}}$. Since $\alpha_1>m-1$, we find
\be
M_n\to0 \m\text{as } n\to\infty.
\ee
Then
\bea \label{k1}
&& k_n^1 \triangleq S_n^{p+q-m+1}M_n^{m-p} = S_n^{\frac{p+mq-(m-p)\alpha_1}{m}}\to 0
 \m\text{ as } n\to\infty \text{ if $\alpha_1>\frac{p+mq}{m-p}$}.
  \eea
Likewise
\be
 k_n^2 \triangleq S_n^{1-m+\alpha_2}M_n^{m-\alpha_2}=S_n^{\frac{\alpha_2(1+\alpha_1)-m\alpha_1}{m}}\to0\ \ \mathrm{as}\ n\to\infty \m\text{ since } \alpha_2 < \frac{m\alpha_1}{\alpha_1+1},
\ee
and
\be
k_n^3 \triangleq  S_n^{1-m}M_n^m=S_n^{-\alpha_1}\to0\ \ \mathrm{as}\ n\to\infty.
\ee
Hence
\be
\label{0.70}
\begin{split}
 |\theta_n(x,w_n,\nabla w_n)|
\leq &
c_0 k_n^1 w_n^q|\nabla w_n|^p+M_1 w_n^{\alpha_1}+M_2k_n^2|\nabla w_n|^{\alpha_2} +\lambda_nk_n^3\\
\le & C(|\nabla w_n|^p+|\nabla w_n|^{\alpha_2} +1).
\end{split}
\ee
Since $0<p<m, 0<\alpha_2<m$,  notice also that
$\|w_n\|_{L^\infty(\Omega_n)}\le 1$, the classical work \cite{L1} implies the $C^{1,\tau}$ regularity result up to boundary,
\begin{equation}\label{0.75}
\|\nabla w_n\|_{C^{1+\tau}(\bar \Omega_n)}\leq C\ \ \ \mathrm{in}\ \bar\Omega_n 
\end{equation}
for certain $C>0$ independent of $n$. Owing to \eqref{0.74}, \eqref{0.75} and mean value theorem, we have
$$1= w_n(M_n^{-1}\delta_n \vec e)-w_n(0)\leq\|\nabla w_n\|_{\infty}M_n^{-1}\delta_n\leq C M_n^{-1}\delta_n.$$
The proof of Step 1 is complete.

{\bf Step 2.} Further estimates on $\theta_n$.
We have assumed $f(x,u,\beta) = f_1(x,u,\beta) + f_2(x,u)+ f_3(x,\beta)  $; here, $x, u, \beta$ are regarded as independent variables. We assume that $b=1$. The case $b=0$ is trivial. For this step, we do not need to assume that $f_1$ is non-negative -- it is required only for the Harnack inequality in the next step. For $u\ge 0$,
\bea
   && |f_1(x,u,\beta)| \le c_0 |u|^q |\beta|^p, \\
   && \left\{
   \begin{split}&u^{\alpha_1} \le f_2(x,u) \le M_1 u^{\alpha_1} , \\
   &
   |f_2(x_1,u)-f_2(x_2,u)| \le \omega(|x_1-x_2|) u^{\alpha_1}, \m \omega(0+)=0,\\
     & u\frac{\p f_2(x,u)}{\p u} \le \g_2 f_2(x,u) \m\text{for some } \gamma_2\in\left(0,\frac{N(m-1)+m}{N-m}\right), \\
  & u^2 \left|\frac{\p^2 f_2(x,u)}{\p u^2} \right| \le C f_2(x,u),
   \end{split} \right.\label{B2} \\
   && |f_3(x,\beta)| \le M_2 |\beta|^{\alpha_2}.
\eea
Writing for the corresponding terms: $\theta_n(y, w,\eta) = \theta_{1,n}(y, w,\eta)+\theta_{2,n}(y, w )+\theta_{3,n}(y,  \eta)+ \la_n k_n^3$, we find, for any $K\ge 1$ and any $y\in \bar \Omega_n$, $0\le w\le1$, $|\eta|\le K$,
\bea
&& |\theta_{1,n}(y, w,\eta)|\le c_0 k_n^1w^q|\eta|^p \le c_0k_n^1 K^p\to 0 \text{ as } n\to \infty \text{ if }\alpha_1>\frac{p+mq}{m-p}.\\
&&|\theta_{3,n}(y, \eta)|\le M_2 k_n^2|\eta|^{\alpha_2} \le M_2 k_n^2 K^{\alpha_2} \to 0 \text{ as } n\to \infty .
\eea
For $\theta_{2,n}$, since $M_n\to 0$ and $\omega(0+)=0$, we have
\bea
&& |\theta_{2,n}(y_1, w)-\theta_{2,n}(y_2, w)|\le \omega( M_n|y_1-y_2|) w^{\alpha_1} \to 0 \m\text{ as } n\to \infty, \\
&& w^{\alpha_1} \le \theta_{2,n}(y,w) \le M_1 w^{\alpha_1}, \\
&& w \frac{\p\theta_{2,n}(y, w)}{\p w} = S_n^{1-m}M_n^m \cdot \Big(
\frac{u}{S_n} \Big)\cdot \Big( S_n \frac{\p f_2}{\p u}\Big) \le S_n^{1-m}M_n^m \gamma_2 f_2 = \gamma_2 \theta_{2,n}(y, w), \\
&& w^2 \left| \frac{\p^2 \theta_{2,n}(y, w)}{\p w^2}\right|
\le C \theta_{2,n}(y, w). \label{diffw}
\eea
It follows that, passing to a subsequence if necessary,  on any compact subset of $(y,w,\eta)$,
\be
 \theta_n(y,w,\eta) \to B(w) \text{ as } n\to \infty,
\ee
where $B(w)$ is independent of $y$ and $\eta$, and  \re{diffw} implies that $B(w)$ is continuously differentiable for $w>0$, and
\bea \label{ww}
&& w^{\alpha_1} \le B(w) \le M_1 w^{\alpha_1} , \mm
 w \frac{\p B(w)}{\p w} \le \gamma_2 B(w).
\eea

{\bf Step 3.} Show that $S_n^{\frac{\alpha_1-m+1}{m}}\delta_n\le C$
and hence $\delta_n\to 0$.

Clearly, for any $B_R$, \[\int_{B_R} u^\gamma \ge |B_R| \inf_{B_R} u^\gamma = \frac{\pi^{n/2} R^N}{\Gamma(1+n/2)}  \Big(\inf_{B_R} u\Big)^\gamma.  \]
Using also \re{5.2} of  Lemma \ref{lem5.1} and the Harnack inequality Theorem \ref{thm:5.3},
taking $R= \delta_n/2$ so that $B_{2R}(x_n) \subset \Omega$, we find
\be
 S_n \le C \Big[\inf_{B_R} u + R^m \lambda_n \Big]
  \le C\Big[ \Big(R^{-N}\int_{B_R} u^\gamma\Big)^{\frac1\gamma} + R^m \lambda_n\Big]
   \le CR^{\frac{-m}{\alpha_1-m+1}},
\ee
and the desired estimates follow.

{\bf Step 4.} Take the blowup limit.

$C^{1+\tau}$ estimates imply the compactness. We also have
$0<c\le M_n^{-1}\delta_n\le C$. Passing to a subsequence if necessary, we find that
$w_n \to w$ on any compact subset of $\bar { \mathbb{R}^N_+} \cap \bar \Omega_n$, $M_n^{-1}\delta_n\to \eta>0$, and $w$ satisfies,
\be
 \left\{\begin{split}
 & w(y) = 0 , \m y \in \partial \mathbb{R}^N_+, \\
 & 0 \le w(y) \le 1 = w(\eta \vec e), \m y\in \mathbb{R}^N_+,  \m \eta>0.
 \end{split}\right.
\ee
and 
\bea
 & \label{N1} - \Delta_m w  =  B(w),  \m y\in \mathbb{R}^N_+, \m\text{ if } \alpha_1>\frac{p+mq}{m-p},
 \eea
 where $B(w)$ satisfies \re{ww}, which is a contradiction to Theorem \ref{thm:6.1}.
\end{proof}

\begin{rem} Here are some observations:
\begin{enumerate}
\item
For the case $\alpha_1 = \frac{p+mq}{m-p}$, we have $k_n^1 \equiv 1$ in \re{k1}.
 A separate Liouville theorem on a half space is needed for this situation in order to carry out the proof.
\item If we take $f_2(x,u)=u^{\alpha_1}$, then $\alpha_1$ will need to satisfy $\alpha_1 < \frac{N(m-1)+m}{N-m}$, but we actually need $\alpha_1 < \frac{N(m-1)}{N-m}$ in other parts of the proof.
\end{enumerate}
\end{rem}

\subsection{Existence of positive weak solutions}

\paragraph{\ \ \ This subsection is devoted to prove existence of positive weak solutions of \eqref{P}. The proof is based on the following fixed point theorem.}
\begin{thm}\label{thm:7.1}(See \cite[Theorem 4.1]{R})
Let $K$ be a cone in a Banach space and $\psi:K\to K$ a compact operator such that $\psi(0)=0$. Assume that there exists $r>0$, verifying\\
$\mathrm{(a)}\ u\neq t\psi(u)\ \mathrm{for\ all\ }\|u\|=r,\ t\in[0,1].$\\
Assume also that there exist a compact homotopy $H:[0,1]\times K\to K$, and $R>r$ such that\\
$\mathrm{(b1)}\ \psi(u)=H(0,u)\ \mathrm{for\ all}\ u\in K.$\\
$\mathrm{(b2)}\ H(t,u)\neq u\ \mathrm{for\ any}\ \|u\|=R,\ t\in[0,1].$\\
$\mathrm{(b3)}\ H(1,u)\neq u\ \mathrm{for\ any}\ \|u\|\leq R.$\\
Let $D=\left\{u\in K:r<\|u\|<R\right\}.$ Then $\psi$ has a fixed point in $D$.
\end{thm}

Combining Theorem \ref{thm:7.1}, Propositions \ref{prop6.1} and \ref{prop6.2}, we will give the proof of the existence results.
\begin{proof}[Proof of Theorem \ref{thm:E}]
We use Theorem \ref{thm:7.1} to prove the existence results. Let $K\subset C^{1,\tau}(\overline{\Omega})$ be the subset of nonnegative functions. Clearly, $K$ is a cone. Define
$$\psi=T\circ N:C^{1,\tau}(\overline{\Omega})\to C^{1,\tau}(\overline{\Omega}),$$
where  
$$N:C^{1,\tau}(\overline{\Omega})\to L^\mu({\Omega}), \; (1 \ll \mu<\infty),\m N(u)=f(x,u,\nabla u).$$
Under our assumptions, $N$ actually maps a bounded set in $C^1(\bar\Omega)$ to a $L^\infty(\Omega)$ bounded set $\{ \|u\|_{L^\infty(\Omega)}\le K\}$.
 We also have, for a zero measure set $\mathscript{Z}\subset \Omega$,
\be
 f \text{ is  continuous on } (\Omega\setminus \mathscript{Z})\times R \times R^N.
\ee
In particular, $f$ is a Carath\'{e}odory function. It follows that the operator $N$ is continuous. Using also the fact that $C^{1,\tau}(\overline{\Omega})$ is compactly embedded in $C^1(\overline{\Omega})$, we derive that $N$ is compact. For any $v\in L^\mu( {\Omega})$
such that $\|v\|_{L^\infty(\Omega)}\le K$,  $T(v)\in C^{1,\tau}(\Omega)$ is the unique weak solution of the equation
\begin{equation}\label{7.1}
\left\{\begin{array}{ll}
-\Delta_mT(v)=v,&x\in\Omega,\\
v(x)=0,&x\in\partial\Omega.
\end{array}\right.
\end{equation}
It follows by the standard compactness and uniqueness argument that $T$ is a continuous operator. Thus, $\psi$ is compact.

First, we verify condition (a) by contradiction. Assume there exists $u\in K\backslash\{0\}$ such that $u=t\psi(u)$ for certain $t\in[0,1]$. Applying \eqref{7.1} with $v=N(u)$, we obtain
\begin{equation}\label{7.2}
-\Delta_m\psi(u)=f(x,u,\nabla u),\ \ \ x\in\Omega.
\end{equation}
Then $u=t\psi(u)$ satisfies
\begin{equation}\label{0.8}
-\Delta_mu=t^{m-1}f(x,u,\nabla u),\ \ \ x\in\Omega.
\end{equation}
Multiplying the equation \eqref{0.8} by $u$ and integrating, and applying the H$\ddot{\mathrm{o}}$lder inequality and the Sobolev embedding inequality, we obtain
\begin{equation}\label{0.4}
\begin{split}
\int_{\Omega}|\nabla u|^m=&t^{m-1}\int_{\Omega}f(x,u,\nabla u)u\\
\leq&\int_{\Omega}\left[c_0u^{q+1}|\nabla u|^p+M_1u^{\alpha_1+1}+M_2u|\nabla u|^{\alpha_2}\right]\\
\leq&c_0\left[\int_{\Omega}|\nabla u|^{p\cdot\frac{m}{p}}\right]^{\frac{p}{m}}\left[\int_{\Omega}u^{\frac{m(q+1)}{m-p}}\right]^{\frac{m-p}{m}}+M_1 \int_{\Omega}  u^{\alpha_1+1} \\
&+M_2\left[\int_{\Omega}|\nabla u|^{\alpha_2\cdot\frac{m}{\alpha_2}}\right]^{\frac{\alpha_2}{m}}\left[\int_{\Omega}u^{\frac{m}{m-\alpha_2}}\right]^{\frac{m-\alpha_2}{m}}\\
\leq& C\Bigg\{\left[\int_{\Omega}|\nabla u|^{m}\right]^{\frac{p+q+1}{m}}+ \left[\int_{\Omega}|\nabla u|^{m}\right]^{\frac{\alpha_1+1}{m}}+ \left[\int_{\Omega}|\nabla u|^{m}\right]^{\frac{\alpha_2+1}{m}} \Bigg\},
\end{split}
\end{equation}
\void{\blue Since the Harnack inequality is valid for this system, and the Harnack inequality is established by iterative Sobolev's embedding, we must have
 $\frac{m}{m-\alpha_2}\le m^*, \frac{m(q+1)}{m-p}\le m^*$. It follows that the
 last inequality holds. As a matter of fact, if these two inequalities cannot
 derived, then some proof must be wrong.}
 the last inequality holds if we verify that $\alpha_1+1\le m^*$, $\frac{m}{m-\alpha_2}\le m^*, \frac{m(q+1)}{m-p}\le m^*$. Indeed, since $\alpha_1<\frac{N(m-1)}{N-m}$, we have $\alpha_1+1<\frac{N(m-1)}{N-m}+1 =\frac{m(N-1)}{N-m}< m^*$. Similarly,
\[
\alpha_2<\frac{m\alpha_1}{\alpha_1+1}
= m-\frac{m}{\alpha_1+1}
< m-\frac{N-m}{N-1},
\]
so that
\[
\frac{m}{m-\alpha_2}<\frac{m(N-1)}{N-m}< m^*.
\]
By \re{F}, we have
\beaa
\frac{m(q+1)}{m-p} - m^*
 & = &\frac{m[q(N-m)-N(m-1)+pN-m]}{(N-m)(m-p)} \\
 & < &\frac{m[q(N-m)-N(m-1)+p(N-1)]}{(N-m)(m-p)}<0.
\eeaa
On the one hand, since $\frac{p+q+1}{m}>\frac{m+q
}{m} \ge1$, $\frac{\alpha_1+1}{m}>\frac{(m-1)+1}{m}=1$ and $\frac{\alpha_2+1}{m}>\frac{(m-1)+1}{m}=1$, it follows from \eqref{0.4} that
\begin{equation}\label{7.3}
\int_{\Omega}|\nabla u|^m\geq C_1.
\end{equation}
for a certain $C_1>0$. On the other hand, choosing $r>0$ sufficiently small such that $\|u\|_{C^{1,\tau}}\leq r$ would imply
\begin{equation*}
\int_{\Omega}|\nabla u|^m\leq \int_{\Omega}\|u\|_{C^{1,\tau}}^m\leq|\Omega|r^m < C_1,
\end{equation*}
which is a contradiction with \eqref{7.3}.

Second, we denote a compact homotopy
$$H:[0,1]\times K\to K,\ H(t,u)=T\circ\left[N(u)+t\lambda_0\right],$$
where $\lambda_0$ is given by Proposition \ref{prop6.1}. Condition (b1) is clearly true. Next, we verify condition (b2) by contradiction. Assume there exists $u\in K\backslash\{0\}$ such that $u=H(t,u)$. Applying \eqref{7.1} with $v=N(u)+t\lambda_0$, we obtain
\begin{equation*}
-\Delta_mu=f(x,u,\nabla u)+t\lambda_0,\ \ \ x\in\Omega.
\end{equation*}
It follows from $t<1$ that $t\lambda_0<\lambda_0$. Proposition \ref{prop6.2} implies that $\|u\|_{\infty}<C$. Using the $C^{1,\tau}$ estimate in \cite{L1}, we can find $R>0$ such that $\|u\|_{C^{1,\tau}}<R$, which is a contradiction with $\|u\|_{C^{1,\tau}}=R$. Similarly, in the validation of condition (b3), we obtain
\begin{equation}\label{0.5}
-\Delta_mu=f(x,u,\nabla u)+\lambda_0,\ \ \ x\in\Omega.
\end{equation}
Due to Proposition \ref{prop6.1}, problem \eqref{0.5} has no positive solution, which is a contradiction.

Last, Theorem \ref{thm:7.1} implies that there exists a positive weak solution $u$ such that $\psi(u)=u$. It follows from \eqref{7.2} that $u$ satisfies
\begin{equation*}
-\Delta_mu=f(x,u,\nabla u),\ \ \ x\in\Omega,
\end{equation*}
that is, \eqref{P} admits at least one positive solution. The proof is complete.
\end{proof}

\section{Weak Harnack inequality   }
The following theorem describes the weak Harnack inequality, the proof is referred in \cite{T1}.
\begin{thm}\label{thm:5.1}(See  \cite{T1})
Let $u$ be a nonnegative weak solution of inequality
$$ - \Delta_mu\ge 0\ \ \mathrm{in}\ \Omega.$$
Take $\gamma\in\left[1, m_*-1\right)$, where $m_* =\frac{m(N-1)}{N-m}$ and $R>0$ such that $B_{3R}\subset\Omega$. Then there exists a positive constant $C=C(N,m,\gamma)$(independent of $R$) such that
\begin{equation*}
\inf_{B_R}u\geq CR^{-\frac{N}{\gamma}}\|u\|_{L^{\gamma}(B_{2R})}.
\end{equation*}
\end{thm}
In the original Trudinger's paper \cite{T1}, $B_{3R}\subset \Omega$ was assumed. However, the theorem stated in \cite{FL1,R} contains an typographical error which required only $B_{2R}\subset \Omega$. A simple counter example is essentially the $x_1$ derivative of the fundamental solution of the  Laplacian, given by
$ u = \frac{x_1}{|x|^{N}}$. Let $B_R = \{|x-(1,0,\cdots,0)|<R\}$.
Then $u$ is positive on $B_1$ and $\Delta u = 0$.  Clearly
$\|u\|_{L^{\gamma}(B_{1})} = \infty$ for any $\gamma\ge \frac{N}{N-1}$, but
$m_*-1=2_*-1=\frac{N}{N-2}>\frac{N}{N-1}$, and $\inf_{B_{1/2}}u$ is clearly finite.
Given this, it seems more details are needed to show $B_{3R}\subset \Omega$ in the blowup argument for obtaining
the $L^\infty$ bounds in \cite{R}.

\section*{Acknowledgments}{This work was partially supported by the National Natural Science Foundation of China (No. 12071044). }

\end{CJK}
\end{document}